\documentclass[sn-mathphys]{sn-jnl}

\jyear{2021}%

\usepackage{multirow}
\usepackage{float}
\allowdisplaybreaks[2]

\usepackage{graphicx}
\usepackage{subfigure}
\usepackage{setspace}
\usepackage{latexsym}
\usepackage{bbm}
\usepackage{color}
\usepackage[misc]{ifsym}                     
\usepackage{}
\usepackage{amsmath,amssymb,amsfonts}
\usepackage{subfigure}
\usepackage{array}
\usepackage{mathrsfs}
\usepackage{mathptmx}
\usepackage{empheq}
\usepackage{diagbox}
\usepackage{amsmath}
\usepackage{dsfont}
\usepackage{amsthm}
\usepackage{graphicx}
\usepackage{grffile}
\usepackage{subfigure}
\usepackage[justification=centering]{caption}
\usepackage{color}
\usepackage{array}
\usepackage{mathrsfs}
\usepackage{multirow}
\usepackage{latexsym}
\usepackage{bbm}
\usepackage{booktabs}
\usepackage[flushleft]{threeparttable}
\usepackage{multirow}
\usepackage{mathptmx}
\usepackage{graphicx}
\usepackage{empheq}
\usepackage{caption}
\usepackage{commath}
\usepackage{algorithm}
\usepackage{algpseudocode}
\usepackage{float}
\newfloat{algorithm}{H}{loa}

\newtheorem{example}{Example}

\newtheorem{Theo}{Theorem}

\newtheorem{Rem}{Remark}
\newtheorem{Lem}{Lemma}

\newtheorem{Coro}{Corollary}

\numberwithin{equation}{section}
\numberwithin{Lem}{section}
\numberwithin{Defi}{section}
\numberwithin{Theo}{section}
\numberwithin{Pro}{section}
\numberwithin{Rem}{section}
\numberwithin{Coro}{section}
\numberwithin{Fig}{section}

\begin{document}

\title[SParareal algorithm for SDEs ]{Stochastic Parareal Algorithm for Stochastic Differential Equations }

\author[1]{\fnm{Huanxin} \sur{Wang}}\email{wanghx02022@163.com}

\author[2]{\fnm{Junhan} \sur{LYU}}\email{23213078@life.hkbu.edu.hk}

\author[3]{\fnm{Zicheng} \sur{Peng}}\email{1220012523@student.must.edu.mo}


\author*[1]{\fnm{Min} \sur{Li}}\email{liminmaths@163.com}


\affil[1]{\orgdiv{School of Mathematics and Physics}, \orgname{China University of Geosciences}, \orgaddress{\city{Wuhan}, \postcode{430074},  \country{China}}}

\affil[2]{\orgdiv{School of Business}, \orgname{Hong Kong Baptist University}, \orgaddress{\city{Kowloon Tong}, \postcode{Hong Kong  SAR},  \country{China}}}

\affil[3]{\orgdiv{School of Business}, \orgname{Macau University of Science and Technology}, \orgaddress{\city{ Macau}, \postcode{999078},  \country{China}}}

\affil*[1]{\orgdiv{School of Mathematics and Physics \emph{and} Center for Mathematical Sciences}, \orgname{China University of Geosciences}, \orgaddress{\city{Wuhan}, \postcode{430074}, \country{China}}}


\abstract{
 This paper analyzes  the SParareal algorithm for stochastic differential equations (SDEs). Compared to the classical Parareal algorithm, the SParareal algorithm accelerates convergence by introducing stochastic perturbations, achieving linear convergence over unbounded time intervals. We first revisit the classical Parareal algorithm and stochastic  Parareal algorithm. Then we investigate mean-square stability of the SParareal algorithm based on the stochastic $\theta$-method for SDEs, deriving linear error bounds under four sampling rules. Numerical experiments demonstrate the superiority of the SParareal algorithm in solving both linear and nonlinear SDEs,  reducing the number of iterations required compared to the classical Parareal algorithm. 
}





\maketitle

\section{Introduction}
The rise of large-scale parallel computing has revolutionized computational sciences, enabling high-resolution simulations across diverse fields such as physics, astronomy, meteorology, biology, and finance. This advancement has intensified the demand for efficient methods to solve time-dependent problems. By leveraging multi-core processors and distributed architectures, time parallelization has emerged as a promising strategy to accelerate computations for complex systems.


Among the principal time-parallel methods, the Parareal algorithm has gained significant attention for its rapid convergence, computational efficiency, and ease of implementation. First introduced by Lions, Maday, and Turini in 2001 \cite{2001Traffic}, the Parareal algorithm decomposes the time domain into discrete intervals, computing the solution in parallel by coupling coarse (low-resolution) and fine (high-resolution) approximations. This strategy reduces the computational burden of solving large systems of time-dependent ordinary differential equations (ODEs), making it particularly suitable for high-performance computing applications.


Extensive theoretical and numerical analyses have been conducted on the Parareal algorithm. For instance, Gander and Vandewalle \cite{2007Traffic} demonstrated that Parareal can be interpreted as a variant of the multiple shooting method, showing superlinear convergence on bounded time intervals and linear convergence on unbounded ones. Subsequent studies have further explored its stability and convergence properties, extending its applicability to complex problem classes \cite{2005Traffic, 2007Traffic}. The algorithm has been successfully applied in various fields, including fluid-structure interaction \cite{2010Time}, Navier-Stokes equations \cite{2005A}, and reservoir simulations \cite{2005B}.


Despite its wide applicability, the Parareal algorithm faces challenges, particularly in the number of iterations required to achieve convergence. To address this limitation, recent studies have proposed stochastic extensions to enhance its performance. Pentland et al. \cite{2021Traffic} introduced the stochastic Parareal (SParareal) algorithm, which samples initial values from dynamically evolving probability distributions across time intervals. Their results show that increasing the number of samples accelerates convergence compared to the classical Parareal algorithm for various ODE systems. In follow-up work, Pentland et al. \cite{2022Traffic} established theoretical bounds on the mean square errors of the SParareal algorithm, further validating its improved performance for both linear and nonlinear ODEs.


The application of the Parareal framework to stochastic differential equations (SDEs) is a relatively recent development. Wu et al. \cite{Wu2011} extended the Parareal algorithm to SDEs and analyzed its mean square stability, providing inspiration for our research. Building on their work, this study aims to enhance the Parareal framework by introducing the SParareal algorithm for SDEs. Using the stochastic \(\theta\)-method for numerical integration, we provide a theoretical analysis of the algorithm's convergence properties.


While prior research has adapted the Parareal framework to SDEs and explored numerical aspects \cite{doi:10.1137/080733723, Engblom2008Parallel, SUBBER2018190}, our work seeks to extend this approach by introducing and rigorously analyzing the stochastic Parareal (SParareal) method. Specifically, we investigate its theoretical properties, focusing on convergence.


The remainder of this paper is organized as follows: Section 2 introduces the SParareal algorithm and its sampling strategy. Section 3 investigates the mean square convergence of the SParareal method when applied to SDEs. In Section 4, numerical experiments are conducted to validate the theoretical results, with a particular focus on demonstrating the advantages of the stochastic Parareal algorithm over the classical Parareal algorithm.

\section{Preparation}
Consider the following stochastic differential equation (SDE):
\begin{equation}
    du(t) = f(u(t)) \, dt + g(u(t)) \, dW(t),
    \label{2.1}
\end{equation}
where $f(u(t))$ and $g(u(t))$ are given functions. We define $u(t)$ as a solution to this SDE. In this context, $f(u(t))$ is the drift term and $g(u(t))$ is the diffusion term, with $t \in [0, T]$ and $u_0$ representing the initial condition at $t = 0$.

\subsection{The Parareal Algorithm}

The classical Parareal algorithm {(referred to as $P$)} provides a numerical solution to equation \eqref{2.1} using two distinct propagation operators, denoted as $\mathcal{G}_{\Delta T}$ and $\mathcal{F}_{\Delta t}$. The operator $\mathcal{G}_{\Delta T}$ is the coarse-step propagation operator, typically implemented using a low-order, computationally inexpensive method, while $\mathcal{F}_{\Delta t}$ is the fine-step propagation operator, generally implemented using a higher-order, more computationally expensive method.

The numerical formulation of the Parareal algorithm is provided as follows: First, we partition the entire time interval $[0, T]$ into $N$ subintervals, each denoted by $[T_n, T_{n+1}]$, for $n = 0, 1, \dots, N-1$. We assume all subintervals have uniform size, i.e., $\delta T = \frac{T}{N} = T_{n+1} - T_n$. Next, entire time interval $[0, T]$ is subdivided into smaller subintervals as follows:
\begin{itemize}
    \item entire time interval $[0, T]$ is divided into $N_g$ equal parts, each of size $\Delta T = \frac{T}{N_g}$.
    \item entire time interval $[0, T]$ is divided into $N_f$ equal parts, with each subinterval having size $\Delta t = \frac{T}{N_f}$.
    \item {$N_f$ and $N_g$ is a multiple of $N$, let $M=\frac{N_f}{N_g}=\frac{\Delta T}{\Delta t}$.}
\end{itemize}

At time $T_n$, the numerical solution after $k$ iterations is denoted as $U^k_n$. The Parareal scheme, based on the definitions of the two numerical propagation operators $\mathcal{G}_{\Delta T}$ and $\mathcal{F}_{\Delta t}$, is given by:
    \begin{align}
        & U^0_0 = u_0, \\
        & U^0_{n+1} = \mathcal{G}_{\Delta T}(U_n^0), \quad 0 \leq n \leq N-1, \\
        & U^{k+1}_{n+1} = \mathcal{G}_{\Delta T}(U_n^{k+1}) + \mathcal{F}_{\Delta t}(U_n^k) - \mathcal{G}_{\Delta T}(U_n^k), \quad 0 \leq n \leq N-1.
        \label{q1}
    \end{align}
\subsection{The SParareal Algorithm}

The SParareal algorithm {(referred to as $P_s$)} extends the classical Parareal algorithm ($P$) by introducing perturbations (i.e., noise) into the numerical solutions, reducing the number of iterations needed for convergence. The core idea of the SParareal algorithm is to generate $m$ samples for each unresolved subinterval $T_n$ within the neighborhood $U^k_n$ of the current predicted correction solution, based on a given probability distribution. These samples are propagated in parallel using the fine-step propagation operator $\mathcal{F}_{\Delta t}$. The SParareal algorithm can be described as follows: 
\begin{figure}[htbp]
    \centering
    \includegraphics[width=1\linewidth]{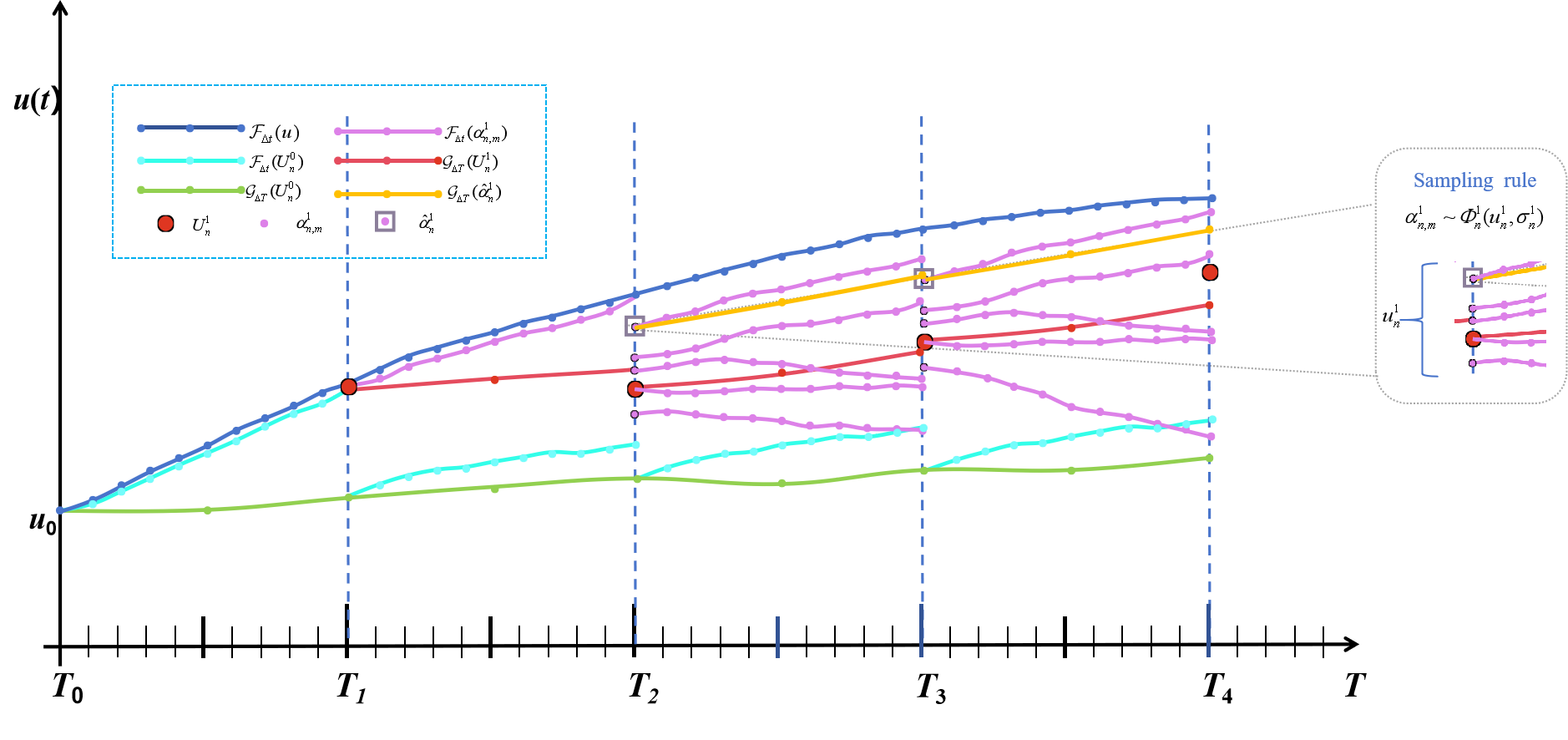}
    \captionsetup{font={footnotesize}}
    \caption{Illustration of the sampling and propagation process of the SParareal algorithm after $k = 1$. The fine solution is represented by the blue line,  the coarse solution at $k = 0$ by the green line, the fine solution at $k = 0$ by the cyan line, the coarse solution at $k = 1$ by the red line, and the predicted corrected solution at $k = 1$ by red dots. For $m = 5$, four samples $\alpha_{n,m}^1$ (purple dots) are drawn from distributions with means $U^1_2$ and $U^1_3$, and some finite standard deviations respectively. The best-selected sample $\widehat{\alpha}^1_n$ is then propagated using $\mathcal{G}_{\Delta T}$ (orange-yellow line).}
     \label{picture1}
\end{figure}

    At time $T_n$, the numerical solution obtained after $k$ iterations is denoted as $U^k_n$. The SParareal scheme is given by the following equations:
    \begin{align}
        & U^0_0 = u_0, \\
        & U^0_{n+1} = \mathcal{G}_{\Delta T}(U_n^0), \quad 0 \leq n \leq N-1, \\
        & U^1_{n+1} = \mathcal{G}_{\Delta T}(U_n^1) + \mathcal{F}_{\Delta t}(U_n^0) - \mathcal{G}_{\Delta T}(U_n^0), \quad 0 \leq n \leq N-1, \\
        & U^{k+1}_{n+1} = \mathcal{G}_{\Delta T}(U^{k+1}_n) + \mathcal{F}_{\Delta t}(U^k_n) - \mathcal{G}_{\Delta T}(U^k_n) + \xi^k_n(U^k_n), \quad 1 \leq k \leq n \leq N-1.
        \label{2.5}
    \end{align}
    Here, $\xi^k_n(U^k_n)$ represents a random perturbation term, which is given by:
    \begin{equation}
        \xi^k_n(U^k_n) = (\mathcal{F}_{\Delta t}(\alpha^k_n) - \mathcal{G}_{\Delta T}(\alpha^k_n)) - (\mathcal{F}_{\Delta t}(U^k_n) - \mathcal{G}_{\Delta T}(U^k_n)).
    \end{equation}
    The term $\alpha^k_n$ represents the stochastic perturbation, introduced by the random variable. Thus, the correction term in the SParareal algorithm is given by:
    \begin{equation}\label{210}
        U^{k+1}_{n+1} = \mathcal{G}_{\Delta T}(U^{k+1}_n) + \mathcal{F}_{\Delta t}(\alpha^k_n) - \mathcal{G}_{\Delta T}(\alpha^k_n),\quad 0 \leq k \leq n \leq N-1.
    \end{equation}
    
The random variable $\alpha^k_n$ follows a probability distribution as defined by the sampling rules in Table \ref{tab:my_label}. Sampling rules 1 and 2 correspond to multivariate Gaussian perturbations with marginal means $\mathcal{F}_{\Delta t}(U^{k-1}_{n-1})$ and $U^k_n$. The marginal standard deviations is given by $\sigma^k_n=\lvert  \mathcal{G}_{\Delta T}(U^k_{n-1})-\mathcal{G}_{\Delta T}(U^{k-1}_{n-1}) \rvert$. The variable $z^k_n \sim \mathcal{N}(0, \mathcal{I}_d)$ represents a standard $d$-dimensional Gaussian random vector.
Sampling rules 1 and 3 share  the same marginal means and standard deviations $\sigma_n^k$. Similarly, sampling rules 2 and 4 have identical marginal means and standard deviations $\sigma_n^k$.
 The $\omega^k_n \sim \mathcal{U}([0,d]^d)$ is a standard $d$-dimensional uniform random vector with independent components. The design of the Sampling rules allows $\alpha^k_n$ to vary with both the iteration index $k$ and the time step $n$. Different perturbations with distinct families of distributions, marginal means, and correlations are constructed to assess the performance of the SParareal algorithm.
Notably, when $k=N$, we have $\alpha^k_n = U^k_n$, which implies that $\xi^k_n(U^k_n) = 0$.

\begin{table}[h]
\renewcommand{\arraystretch}{1.65}
\belowrulesep=0pt
\aboverulesep=4pt
    \centering
    \caption{Sampling Rules for Stochastic Perturbations}
    \begin{tabular}{|c|c|}
    \toprule
        \hline
        \textbf{Sampling Rule} & \(\alpha^k_n\) \\ \hline
        1 & \(\mathcal{F}_{\Delta t}(U^{k-1}_{n-1}) + \sigma^k_n \circ z^k_n\) \\ \hline
        2 & \(U^k_n + \sigma^k_n \circ z^k_n\) \\ \hline
        3 & \(\mathcal{F}_{\Delta t}(U^{k-1}_{n-1}) + \left( \sqrt{3} \sigma^k_n \circ (2 \omega^k_n - 1) \right)\) \\ \hline
        4 & \(U^k_n + \left( \sqrt{3} \sigma^k_n \circ (2 \omega^k_n - 1) \right)\) \\ \hline
        \bottomrule
    \end{tabular}
    \label{tab:my_label}
\end{table}



The operational process of the SParareal algorithm, originally proposed in \cite{2021Traffic} for deterministic differential equations, is outlined in Algorithm \ref{alg:parareal_sampling}. In this work, we extend the algorithm to stochastic differential equations (SDEs).


\begin{algorithm}[htp]
\caption{Stochastic Parareal}
\label{alg:parareal_sampling}
\textbf{Input}: 
    Initial solution \( U^0 \) over time interval \([0, T]\);
  Coarse propagator \( \mathcal{G}_{\Delta T} \);
 Fine propagator \( \mathcal{F}_{\Delta t} \);
    Maximum iteration count \( N \);
  Tolerance \( \varepsilon \);
 Number of samples \( M \);

\begin{algorithmic}[1]
\State Run initial parareal iteration for \( k = 1 \) using standard parareal steps.
\For{$k = 2$ to $N$}
    \State \textbf{Compute Correlation Matrices (if dimension \( d > 1 \)):}
    \Statex \quad Set \( R^{k-1}_n = I \) for all \( n \) if \( k < 3 \).
    \If{$k \geq 3$}
        \For{$n = I+1$ to $N-1$}
            \State Calculate \( R^{k-1}_n \) based on previous samples \( \mathcal{F}(\alpha^{k-2}_{n-1,1}), \ldots, \mathcal{F}(\alpha^{k-2}_{n-1,M}) \).
        \EndFor
    \EndIf

    \State \textbf{Initial Value Sampling and Propagation (Parallel for each sub-interval):}
    \For{$n = I$ to $N-1$}
        \If{$n == I$}
            \State Propagate converged solution from \( T_I \): \( \tilde{U}^{k-1}_{n+1} = \mathcal{F}(U^{k-1}_n) \).
        \Else
            \State Initialize first sample as predictor-corrector: \( \alpha^{k-1}_{n,1} = U^{k-1}_n \).
            \State Propagate \( \alpha^{k-1}_{n,1} \): \( \tilde{U}_{n+1,1} = \mathcal{F}(\alpha^{k-1}_{n,1}) \).
            \For{$m = 2$ to $M$}
                \State Sample \( \alpha^{k-1}_{n,m} \sim \pi^{k-1}_n \) (from defined probability distribution).
                \State Propagate \( \alpha^{k-1}_{n,m} \): \( \tilde{U}_{n+1,m} = \mathcal{F}(\alpha^{k-1}_{n,m}) \).
            \EndFor
        \EndIf
    \EndFor

    \State \textbf{Optimal Sample Selection for Continuity:}
    \For{$n = I+1$ to $N-1$}
        \State Select \( J = \arg \min_{j \in \{1, \ldots, M\}} \|\alpha^{k-1}_{n,j} - \tilde{U}^{k-1}_n\|_2 \).
        \State Set \( \widehat{\alpha}^{k-1}_n = \alpha^{k-1}_{n,J} \), \( \tilde{U}^{k-1}_{n+1} = \tilde{U}_{n+1,J} \).
    \EndFor

    \State \textbf{Coarse Propagation with Selected Optimal Sample (Parallel):}
    \For{$n = I+1$ to $N-1$}
        \State Compute coarse solution: \( \widehat{U}^{k-1}_{n+1} = \mathcal{G}(\widehat{\alpha}^{k-1}_n) \).
    \EndFor

    \State \textbf{Predictor-Corrector Update:}
    \For{$n = I+1$ to $N$}
        \State Perform correction: \( U^k_n = \tilde{U}^{k-1}_n + \widehat{U}^k_n - \widehat{U}^{k-1}_n \).
    \EndFor

    \State \textbf{Convergence Check:}
    \State Update \( I = \max\{n \in \{I+1, \ldots, N\} : \|U^k_i - U^{k-1}_i\|_{\infty} < \varepsilon \text{ for all } i \leq n\} \).
    \If{$I == N$}
        \State Return \( k \) and \( U^k \) as the converged solution.
    \EndIf
\EndFor
\end{algorithmic}

\textbf{Output}: Converged solution \( U^k \) for all sub-intervals \([T_n, T_{n+1}]\).
\end{algorithm}

To prepare for subsequent work, we will present the following lemma.

\begin{Lem}\label{Gronwall_inequality}
    {\cite{2022Traffic}} \label{lemma}
     Let $\widehat{\epsilon}^{(k)}$ be a non-negative sequence, where $\tilde{A}$ and $\tilde{B} \in \mathbb{R}$ are non-negative constants. Suppose the following condition is satisfied:
     \begin{equation}\label{2.9}
        \widehat{\epsilon}^{(k+1)} \leq \tilde{A} \widehat{\epsilon}^{(k)} + \tilde{B} \widehat{\epsilon}^{(k-1)}.
    \end{equation}
    If $\widehat{\epsilon}^{(1)}$ and $\widehat{\epsilon}^{(0)}$ exist, then the following relation holds:
    \begin{equation}\label{2.10}
       \widehat{\epsilon}^{k} \leq \widehat{\epsilon}^{0} \left( \frac{\tilde{A} + \sqrt{\tilde{A}^2 + 4\tilde{B}}}{2} \right)^k.
    \end{equation}
\end{Lem}

\section{Mean Square Stability}
In this section, we consider the SParareal algorithm for solving the Dahlquist test equation and perform a theoretical analysis.

\subsection{Stochastic $\theta$-Method}
We take the drift term $f(u(t)) = \lambda u(t)$ and the diffusion term $g(u(t)) = \mu u(t)$ from equation \eqref{2.1}. Next, we consider the following stochastic differential equation (SDE):
\begin{equation}\label{3.1}
    du(t) = \lambda u \left(t\right) \, dt + \mu u \left(t\right) \, dW \left(t\right), \quad t \in [0,T],
\end{equation}
where $u(0) = u_0 \neq 0$, $\lambda, \mu \in \mathbb{C}$, and $W(t)$ is a standard Wiener process.

Given a parameter $\theta \in [0,1]$, the corresponding stochastic $\theta$-method \cite{MR1781202,MR2204723} for the SDE \eqref{3.1} takes the form:
\begin{equation}
    U_{n+1} = \left( \frac{\left(1 + \left(1-\theta\right)\right) \delta t \lambda}{1 - \theta \delta t \lambda} + \frac{\sqrt{\delta t} \mu}{1 - \theta \delta t \lambda} \mathcal{V}_n \right) U_n.
\end{equation}
We use the stochastic $\theta$-method as the propagation operator. Let:
\begin{equation}\label{3.2}
    \begin{aligned}
        a &= \frac{\left(1 + \left(1-\theta_f\right)\right) \Delta t \lambda}{1 - \theta_f \Delta t \lambda}, \quad \quad
        A = \frac{\left(1 + \left(1-\theta_g\right)\right) \Delta T \lambda}{1 - \theta_g \Delta T \lambda}, \\
        b &= \frac{\sqrt{\Delta t} \mu}{1 - \theta_f \Delta t \lambda}, \quad \quad\quad \quad\quad \quad
        B = \frac{\sqrt{\Delta T} \mu}{1 - \theta_g \Delta T \lambda}.
    \end{aligned}
\end{equation}
Thus, through appropriate calculations, the coarse propagation operator $\mathcal{G}_{\Delta T}$ and the fine propagation operator $\mathcal{F}_{\Delta t}$ are given by:
\begin{equation}\label{34}
\begin{aligned}
    \mathcal{F}(T_n,U_n^{k},\Delta t)&=\prod_{j=0}^{M-1}\left(a+b\upsilon _{n+\frac{j}{M}}\right)U_{n}^{k},\\
   \mathcal{G}(T_n,U_n^{k},\Delta T)&=(A+BV_n)U_{n}^{k}.
\end{aligned}
\end{equation}
where $\{V_n\}$ and $\{\upsilon_{n + \frac{j}{M}}\}_{j=0,1,\dots,M-1}$ are random variables that follow a standard normal distribution along the same path, satisfying the following relation:
\begin{equation*}
    V_n = \frac{1}{\sqrt{M}} \sum_{j=0}^{M-1} \upsilon_{n + \frac{j}{M}}, \quad n = 0, 1, \dots, N-1.
\end{equation*}
We assume that the coarse propagation operator $\mathcal{G}_{\Delta T}$ and the fine propagation operator $\mathcal{F}_{\Delta t}$ employ the stochastic $\theta$-method with parameters $\theta_g$ and $\theta_f$, respectively. It is well-known that for $n \geq 2$, the numerical scheme of the SParareal algorithm is given by \eqref{210}. Combining \eqref{34} with \eqref{210}, the SParareal scheme is expressed as:
\begin{align}\label{3.4}
     U_{n+1}^{k+1} &= \left(A + B V_n\right) U_n^{k+1} + \left[\prod_{j=0}^{M-1} \left(a + b \upsilon_{n + \frac{j}{M}} \right) - \left(A + B V_n\right) \right] \alpha_n^k, ~~n = 0, 1, \dots, N-1.
\end{align}

\subsection{Stochastic mean-square convergence}
In this section, we consider the special case when
$M = \frac{N_f}{N_g} = \frac{\Delta T}{\Delta t} = 2$. Under this assumption, the following conclusions hold:

\begin{Theo}\label{Thm2.1}
\setstretch{1.25}
    Assume that $M = \frac{\Delta T}{\Delta t} = 2$ and $\lambda, \mu \in \mathbb{C}$ satisfy the condition:
\begin{equation}\label{3.5}
        \mathrm{Im} \left( A \overline{a}^2 + \sqrt{2} B \overline{ab} \right) = 0,
    \end{equation}
    where the variables $a$, $b$, $A$, and $B$ are defined in \eqref{3.2}. Let the sequence $\{u_n\}$ be generated serially through $\mathcal{F}_{\Delta t}$. Define the error sequence $\epsilon_n^k=\mathbb{E}\left[\lvert U^k_n-u_n \rvert^2\right]$
    denotes the expectation operator. The sequences $\{U_n^k\}$ and $\{\alpha_n^k\}$ are determined by the relation in \eqref{3.4}, and the term $\alpha_n^k$ is defined according to sampling rule 2 or 4. 
    Let $\widehat{\epsilon}^{(k)} = \sup\limits_{n} \epsilon_n^k$. If $\alpha < 1$, then the linear error bound of the stochastic differential equation \eqref{3.1} satisfies the following relation:
\begin{equation}\label{3.7}
        \widehat{\epsilon}^{(k)} \leq \widehat{\epsilon}^{(0)} \left[ \frac{\frac{2\beta(1 + 2\gamma)}{1 - \alpha} + \sqrt{\left( \frac{2\beta(1 + 2\gamma)}{1 - \alpha} \right)^2 + 4 \frac{4\beta\gamma}{1 - \alpha}}}{2} \right]^k,
    \end{equation}
    where
    \begin{eqnarray}
        \begin{aligned}
            \gamma&=\lvert A  \rvert^2+\lvert B  \rvert^2,\\
            \alpha&=\lvert A  \rvert^2+\lvert B  \rvert^2+\lvert A\overline{a}^2-\lvert B  \rvert^2-\lvert A  \rvert^2+\sqrt{2}B\overline{ab} \rvert,\\
            \beta&=-2\mathrm{Re}(\overline{A}a^2)+\lvert A  \rvert^2+\lvert B  \rvert^2-2\sqrt{2}\mathrm{Re}(\overline{B}ab)\\
            &+(\lvert a  \rvert^2+\lvert b  \rvert^2)^2+\lvert A\overline{a}^2-\lvert B  \rvert^2-\lvert A  \rvert^2+\sqrt{2}B\overline{ab}\rvert.
        \end{aligned}
        \nonumber
    \end{eqnarray}
\end{Theo}

\begin{proof}
Substituting $M = 2$ into equation \eqref{3.4}, the SParareal scheme becomes
\begin{equation}\label{3.9}
    U_{n+1}^{k+1} = (A + B V_n) U_n^{k+1} + \left[ (a + b \upsilon_{n+\frac{1}{2}})(a + b \upsilon_n) - (A + B V_n) \right] \alpha_n^k.
\end{equation}
Based on the expression in \eqref{3.9}, the error between $U_{n+1}^{k+1}$ and $u_{n+1}$ can be written as
\begin{equation}\label{3.10}
    U_{n+1}^{k+1} - u_{n+1} = (A + B V_n)(U_n^{k+1} - u_n) + \left[ (a + b \upsilon_{n+\frac{1}{2}})(a + b \upsilon_n) - (A + B V_n) \right](\alpha_n^k - u_n).
\end{equation}
Taking the modulus and squaring both sides of equation \eqref{3.10} yields
\begin{equation}\label{3.11}
\begin{split}
    \lvert U_{n+1}^{k+1}-u_{n+1}\rvert ^2 = &\left[ (A + B V_n)(U_n^{k+1} - u_n) + \left( (a + b \upsilon_{n+\frac{1}{2}})(a + b \upsilon_n) - (A + B V_n) \right)(\alpha_n^k - u_n) \right] \cdot \\
    & \left[ \overline{(A + B V_n)(U_n^{k+1} - u_n) + \left( (a + b \upsilon_{n+\frac{1}{2}})(a + b \upsilon_n) - (A + B V_n) \right)(\alpha_n^k - u_n)} \right].
\end{split}
\end{equation}
Since $\{V_n\}$ and $\{\upsilon_{n+\frac{j}{2}}\}$ are random variables that follow a standard normal distribution, the following equalities hold
\begin{equation}\label{e1}
    \sqrt{\Delta T} V_n = W\left(\left(n+1\right)\Delta T\right) - W\left(n \Delta T\right),
    \quad \sqrt{\Delta t} \upsilon_{n+\frac{j}{2}} = W\left(n \Delta T + (j+1)\Delta t\right) - W(n \Delta T + j \Delta t),
\end{equation}
where $W(t)$ denotes a standard Wiener process. Thus, by the properties of the Wiener process, we obtain
\begin{equation}\label{e2}
    \sqrt{\Delta T} V_n = \sum_{j=0}^{1} \sqrt{\Delta t} \upsilon_{n+\frac{j}{2}}, \quad \{V_n\} \sim \mathcal{N}(0,1), \quad \{\upsilon_{n+\frac{j}{2}}\}_{j=0,1} \sim \mathcal{N}\left(0,1\right).
\end{equation}
From \eqref{e2}, the following relations hold
\begin{equation}\label{e4}
\begin{aligned}
    &\mathbb{E}\left(V_n\right) = \mathbb{E}\left(\upsilon_{n+\frac{j}{2}}\right) = 0, \\
    &\mathbb{E}(V_n^2) = \mathbb{E}\left(\upsilon_{n+\frac{j}{2}}^2\right) = 1, \\
    &\mathbb{E}\left(\upsilon_{n+\frac{j}{2}} \upsilon_{n+\frac{i}{2}}\right) = 0, \quad i \neq j, \\
    &\mathbb{E}\left(V_n \upsilon_{n+\frac{j}{2}}\right) = \sqrt{\frac{1}{2}}.
\end{aligned}
\end{equation}
Furthermore, since ${U^k_n}$, ${U^{k+1}_n}$, $\alpha^k_n$, and ${u_n}$ are all $\mathcal{F}_{n\Delta t}$-measurable, from equation \eqref{e4}, the following relations hold (the result for ${U^{k+1}_n}$ is similar)
\begin{equation}
\begin{aligned}
    &\mathbb{E}\left(V_n (U^k_n - u_n)\right) = \mathbb{E}\left((U^k_n - u_n) \mathbb{E}(V_n \mid \mathcal{F}_{n\Delta t})\right) = 0, \\
    &\mathbb{E}\left(V_n (\alpha^k_n - u_n)\right) = \mathbb{E}\left((\alpha^k_n - u_n) \mathbb{E}(V_n \mid \mathcal{F}_{n\Delta t})\right) = 0, \\
    &\mathbb{E}\left(V_n^2 (U^k_n - u_n)\right) = \mathbb{E}\left((U^k_n - u_n) \mathbb{E}(V_n^2 \mid \mathcal{F}_{n\Delta t})\right) = \mathbb{E}(U^k_n - u_n), \\
    &\mathbb{E}\left(V_n^2 (\alpha^k_n - u_n)\right) = \mathbb{E}\left((\alpha^k_n - u_n) \mathbb{E}(V_n^2 \mid \mathcal{F}_{n\Delta t})\right) = \mathbb{E}(\alpha^k_n - u_n).
\end{aligned}
\end{equation}
Similarly, it follows that
\begin{equation}
\begin{aligned}
    &\mathbb{E}\left(\upsilon_{n + \frac{j}{2}}^2 (U^k_n - u_n)\right) = \mathbb{E}(U^k_n - u_n), \\
    &\mathbb{E}\left(\upsilon_{n + \frac{j}{2}}^2 (\alpha^k_n - u_n)\right) = \mathbb{E}(\alpha^k_n - u_n), \\
    &\mathbb{E}\left(V_n \upsilon_{n + \frac{j}{2}} \upsilon_{n + \frac{i}{2}} (U^k_n - u_n)\right) = 0, \quad i \neq j, \\
    &\mathbb{E}\left(V_n \upsilon_{n + \frac{j}{2}} \upsilon_{n + \frac{i}{2}} (\alpha^k_n - u_n)\right) = 0, \quad i \neq j, \\
    &\mathbb{E}\left(\upsilon_{n + \frac{j}{2}} (U^k_n - u_n)\right) = \mathbb{E}\left(\upsilon_{n + \frac{j}{2}} \upsilon_{n + \frac{i}{2}} (U^k_n - u_n)\right) = 0, \quad i \neq j, \\
    &\mathbb{E}\left(\upsilon_{n + \frac{j}{2}} (\alpha^k_n - u_n)\right) = \mathbb{E}\left(\upsilon_{n + \frac{j}{2}} \upsilon_{n + \frac{i}{2}} (\alpha^k_n - u_n)\right) = 0, \quad i \neq j.
\end{aligned}
\end{equation}
Taking the expectation of equation \eqref{3.11}, we obtain
\begin{align}\label{3.12}
    &\mathbb{E}\left[\lvert U_{n+1}^{k+1}-u_{n+1}\rvert^2 \right] \nonumber \\
    &= \mathbb{E}\left[ \left( (A + B V_n)(U_n^{k+1} - u_n) + \left( (a + b \upsilon_{n + \frac{1}{2}})(a + b \upsilon_n) - (A + B V_n) \right)(\alpha_n^k - u_n) \right) \right. \nonumber \\
    &\quad \cdot \left. \overline{ \left( (A + B V_n)(U_n^{k+1} - u_n) + \left( (a + b \upsilon_{n + \frac{1}{2}})(a + b \upsilon_n) - (A + B V_n) \right)(\alpha_n^k - u_n) \right) } \right] \nonumber \\
    &= \left(\lvert A \rvert^2 + \lvert B \rvert^2\right) \mathbb{E}\left[\lvert U_n^{k+1} - u_n\rvert^2\right] \nonumber \\
    &\quad + \left( A \overline{a}^2 - \lvert B \rvert^2 - \lvert A \rvert^2 + \sqrt{2} B \overline{ab} \right) \mathbb{E}\left[(U_n^{k+1} - u_n)(\overline{\alpha_n^k} - \overline{u_n})\right] \nonumber \\
    &\quad + \left( \overline{A} a^2 - \lvert B \rvert^2 - \lvert A \rvert^2 + \sqrt{2} \overline{B} ab \right) \mathbb{E}\left[(\overline{U_n^{k+1}} - \overline{u_n})(\alpha_n^k - u_n)\right] \nonumber \\
    &\quad + \left( - \overline{A} a^2 - A \overline{a}^2 + \lvert A \rvert^2 +\lvert B \rvert^2 - \sqrt{2} \overline{B} ab - \sqrt{2} B \overline{ab} + (\lvert a \rvert^2 + \lvert b \rvert^2 )^2 \right) \mathbb{E}\left[\lvert \alpha_n^k - u_n \rvert^2\right] \nonumber \\
    &= I_1 + I_2 + I_3 + I_4.
\end{align}
    From assumption \eqref{3.5}, we obtain 
    \begin{equation}\label{3.13}
        \begin{aligned}
            I_2+I_3
            &=\left(A\overline{a}^2-\lvert B\rvert^2-\lvert A\rvert^2+\sqrt{2}B\overline{ab}\right)\mathbb{E}\left[2\mathrm{Re}\left((U^{k+1}_n-u_{n})(\overline{\alpha}^k_n-\overline{u}_n)\right)\right]\\
            &\le \lvert A\overline{a}^2-\lvert B\rvert^2-\lvert A\rvert^2+\sqrt{2}B\overline{ab}\rvert \left(\mathbb{E}\left[\lvert U^{k+1}_n-u_{n}\rvert ^2\right] + \mathbb{E}\left[\lvert \alpha ^{k}_n-u_{n}\rvert^2\right]\right).
            \end{aligned}
    \end{equation}
     Substituting equation \eqref{3.13} into \eqref{3.12} results in
    \begin{eqnarray}\label{3.15}
        \mathbb{E}\left[\lvert U_{n+1}^{k+1}-u_{n+1}\rvert ^2\right]\le \alpha\mathbb{E}\left[\lvert U_{n}^{k+1}-u_{n}\rvert ^2\right]+\beta\mathbb{E}\left[\lvert\alpha ^{k}_n-u_{n}\rvert^2\right].
    \end{eqnarray}
    where
    \begin{align*}
        \alpha&=\lvert A\rvert^2+\lvert B\rvert^2+\lvert A\overline{a}^2-\abs{B}^2-\abs{ A}^2+\sqrt{2}B\overline{ab}\rvert, \\
        \beta&=-2\mathrm{Re}(\overline{A}a^2)+\lvert A\rvert^2+\lvert B\rvert^2-\sqrt{2}\mathrm{Re}(\overline{B}ab)+(\lvert a\rvert^2+\lvert b\rvert^2)^2+\lvert A\overline{a}^2-\abs{B}^2-\abs{ A}^2+\sqrt{2}B\overline{ab} \rvert.
    \end{align*}
    Applying sampling rule 2 to term $\mathbb{E}\left[\lvert\alpha ^{k}_n-u_{n}\rvert\right]^2$ in equation \eqref{3.15} results in
    \begin{align}\label{3.16}
       \mathbb{E}\left[\lvert\alpha ^{k}_n-u_{n}\rvert\right]^2
        &=\mathbb{E}\left[\lvert U_{n}^{k}+(\sigma^k_n\circ z^k_n)-u_n \rvert^2\right]\nonumber\\
        &=\mathbb{E}\left[\left((U_{n}^{k}-u_n)+(\sigma^k_n\circ z^k_n)\right)\overline{\left( (U_{n}^{k}-u_n)+(\sigma^k_n\circ z^k_n) \right)}\right]\nonumber\\
        &=\mathbb{E}\left[\lvert U_{n}^{k}-u_n\rvert^2+(U_{n}^{k}-u_n)\overline{(\sigma^k_n\circ z^k_n)}+\overline{(U_{n}^{k}-u_n)}(\sigma^k_n\circ z^k_n)+\lvert \sigma^k_n\circ z^k_n\rvert^2\right]\nonumber\\
        &=\mathbb{E}\left[\lvert U_{n}^{k}-u_n\rvert^2\right]+\mathbb{E}\left[\lvert\sigma^k_n\circ z^k_n\rvert^2\right]+\mathbb{E}\left[(U_{n}^{k}-u_n)\overline{(\sigma^k_n\circ z^k_n)}\right]+\mathbb{E}\left[\overline{(U_{n}^{k}-u_n)}(\sigma^k_n\circ z^k_n)\right]\nonumber\\
        &=\mathbb{E}\left[\lvert U_n^k-u_n\rvert^2\right]+\mathbb{E}\left[\lvert \sigma^k_n\circ z^k_n\rvert^2\right]+\mathbb{E}\left[2\mathrm{Re}(U_{n}^{k}-u_n)\overline{(\sigma^k_n\circ z^k_n)}\right]\nonumber\\
        &\le 2\mathbb{E}\left[\lvert U_n^k-u_n\rvert^2\right]+2\mathbb{E}\left[\lvert\sigma^k_n\circ z^k_n\rvert^2\right]\nonumber\\
        &\le 2\mathbb{E}\left[\lvert U_n^k-u_n\rvert^2\right]+2\mathbb{E}\left[\lvert\sigma^k_n\rvert^2\right]\mathbb{E}\left[\lvert z^k_n\rvert^2\right]\nonumber\\
        &= 2\mathbb{E}\left[\lvert U_n^k-u_n\rvert^2\right]+2\mathbb{E}\left[\lvert \sigma^k_n\rvert^2\right].
    \end{align}
For the term $\mathbb{E}\left[\lvert\sigma^k_n\rvert^2\right]$ in equation \eqref{3.16}, we have
    \begin{align}\label{3.17}
    \mathbb{E}\left[\lvert\sigma^k_n\rvert\right]^2 &=\mathbb{E}\left[\lvert(A+BV_n)(U^k_{n-1}-U^{k-1}_{n-1})\rvert^2\right]\nonumber\\
    &=\mathbb{E}\left[\lvert(A+BV_n)(U^k_{n-1}-u_{n-1}) - (A+BV_n)(U^{k-1}_{n-1}-u_{n-1})\rvert^2\right]\nonumber\\
    &=\mathbb{E}\left[\left((A+BV_n)(U^k_{n-1}-u_{n-1}) - (A+BV_n)(U^{k-1}_{n-1}-u_{n-1})\right)\right.\nonumber\\
    &\quad \cdot \left. \overline{\left((A+BV_n)(U^k_{n-1}-u_{n-1}) - (A+BV_n)(U^{k-1}_{n-1}-u_{n-1})\right)}\right]\nonumber\\
    &\le \mathbb{E}\left[\left(\abs{A}^2 + A \overline{B} V_n + \overline{A} B V_n + \abs{B}^2 V_n^2\right) \lvert U^k_{n-1} - u_{n-1}\rvert^2 \right.\nonumber\\
    &\quad + \left(\abs{A}^2 + A \overline{B} V_n + \overline{A} B V_n + \abs{B}^2 V_n^2\right) \lvert U^{k-1}_{n-1} - u_{n-1}\rvert^2 \nonumber\\
    &\quad + \left(\abs{A}^2 + A \overline{B} V_n + \overline{A} B V_n + \abs{B}^2 V_n^2) \overline{(U^k_{n-1} - u_{n-1})}(U^{k-1}_{n-1}-u_{n-1})\right) \nonumber\\
    &\quad + \left.\left(\abs{A}^2 + A \overline{B} V_n + \overline{A} B V_n + \abs{B}^2 V_n^2\right) (U^k_{n-1} - u_{n-1}) \overline{(U^{k-1}_{n-1}-u_{n-1})} \right]\nonumber\\
    &= \left(\abs{A}^2 + \abs{B}^2\right) \left(\mathbb{E}\left[\lvert U^k_{n-1} - u_{n-1}\rvert^2\right] + \mathbb{E}\left[2\mathrm{Re}( (U^k_{n-1} - u_{n-1}) \overline{(u_{n-1}-U^{k-1}_{n-1})})\right] \right.\nonumber\\
    &\quad \left. + \mathbb{E}\left[\lvert U^{k-1}_{n-1} - u_{n-1}\rvert^2\right]\right)\nonumber\\
    &\le 2\gamma \left(\mathbb{E} \left[\lvert U^k_{n-1} - u_{n-1}\rvert^2 \right] + \mathbb{E}\left[\lvert U^{k-1}_{n-1} - u_{n-1}\rvert^2\right]\right).
\end{align}
    where $\gamma = \abs{A}^2 + \abs{B}^2 $. Substituting the result from equation \eqref{3.17} into equation \eqref{3.16} yields
\begin{equation}\label{3.19}
    \begin{aligned}
        \mathbb{E}\left[\lvert\alpha_n^k - u_n \rvert\right]^2 
        &\le 2 \mathbb{E}\left[\lvert U_n^k - u_n \rvert^2\right] + 4 \gamma \left( \mathbb{E}\left[\lvert U_{n-1}^k - u_{n-1} \rvert^2 \right] + \mathbb{E}\left[\lvert U_{n-1}^{k-1} - u_{n-1} \rvert^2\right] \right).
    \end{aligned}
\end{equation}
 Substituting equation \eqref{3.19} into equation \eqref{3.15} gives
\begin{equation}\label{3.20}
    \begin{aligned}
        \mathbb{E}\left[\lvert U_{n+1}^{k+1} - u_{n+1} \rvert^2\right]
        &\le \alpha \mathbb{E}\left[\lvert U_n^{k+1} - u_n \rvert^2\right] + 2\beta \mathbb{E}\left[\lvert U_n^k - u_n \rvert^2\right] \\
        &\quad + 4\beta \gamma \left( \mathbb{E}\left[\lvert U_{n-1}^k - u_{n-1} \rvert^2\right] + \mathbb{E}\left[\lvert U_{n-1}^{k-1} - u_{n-1}\rvert^2\right] \right).
    \end{aligned}
\end{equation}
Let $\epsilon_n^k = \mathbb{E}\left[\lvert U_{n+1}^{k+1} - u_{n+1} \rvert^2\right]$. Taking the supremum over \( n \) yields
\begin{equation}\label{3.22}
    \begin{aligned}
        \sup_n \epsilon_{n+1}^{k+1} &\le \alpha \sup_n \epsilon_n^{k+1} + 2\beta \sup_n \epsilon_n^k \\
        &\quad + 4\beta \gamma \sup_n \epsilon_{n-1}^k + 4\beta \gamma \sup_n \epsilon_{n-1}^{k-1}.
    \end{aligned}
\end{equation}
Let $\widehat{\epsilon}^{(k)} = \sup\limits_{n}\epsilon_n^k$. From equation \eqref{3.22}, we obtain
\begin{equation}\label{3.23}
    \begin{aligned}
        \widehat{\epsilon}^{(k+1)} &\le \alpha \widehat{\epsilon}^{(k+1)} + 2\beta \widehat{\epsilon}^{(k)} + 4\beta \gamma \widehat{\epsilon}^{(k)} + 4\beta \gamma \widehat{\epsilon}^{(k-1)} \\
        &\le \frac{2\beta\left(1 + 2\gamma\right)}{1 - \alpha} \widehat{\epsilon}^{(k)} + \frac{4\beta \gamma}{1 - \alpha} \widehat{\epsilon}^{(k-1)}.
    \end{aligned}
\end{equation}
It is easy to verify that  $\frac{2\beta\left(1 + 2\gamma\right)}{1 - \alpha}$ and $\frac{4\beta \gamma}{1 - \alpha}$ are nonnegative, thus from Lemma \ref{lemma}, we have
\begin{equation}\label{3.24}
    \widehat{\epsilon}^{(k)} \leq \widehat{\epsilon}^{(0)} \left[ \frac{\frac{2\beta\left(1 + 2\gamma\right)}{1 - \alpha} + \sqrt{\left( \frac{2\beta\left(1 + 2\gamma\right)}{1 - \alpha} \right)^2 + 4 \frac{4\beta \gamma}{1 - \alpha}}}{2} \right]^k.
\end{equation}
\end{proof}

\begin{Rem}
The proof for the sampling rule 4 follows the same reasoning as that for sampling rule 2, with the difference that $\mathbb{E}\left[\lvert\sqrt{3}(2\omega^k_n-1)\rvert^2\right] = 1$ is used in place of $\mathbb{E}\left[\lvert z^k_n\rvert^2\right] = 1$.
\end{Rem}

\begin{Coro}\label{Corollary}
\setstretch{1.25}
    Assume that \( M = \frac{\Delta T}{\Delta t} = 2 \) and that \( \lambda, \mu \in \mathbb{C} \) satisfy the condition:
    \begin{equation}\label{condition}
        \mathrm{Im}\left( A \overline{a}^2 + \sqrt{2} B \overline{ab} \right) = 0,
    \end{equation}
    where the variables $a$, $b$, $A$, and $B$ are defined in \eqref{3.2}. Let the sequence $\{u_n\}$ be generated serially through $\mathcal{F}_{\Delta t}$. Define the error sequence $\epsilon_n^k=\mathbb{E}\left[\lvert U^k_n-u_n \rvert^2\right]$
    denotes the expectation operator. The sequences $\{U_n^k\}$ and $\{\alpha_n^k\}$ are determined by the relation in \eqref{3.4}, and the term $\alpha_n^k$ is defined according to sampling rule 1 or 3. 
    Let $\widehat{\epsilon}^{(k)} = \sup\limits_{n} \epsilon_n^k$. If $\alpha < 1$, then the linear error bound of the stochastic differential equation \eqref{3.1} satisfies the following relation
    \begin{equation}\label{3.26}
        \widehat{\epsilon}^{(k)} \leq \widehat{\epsilon}^{(0)} \left[ \frac{ \frac{4\beta \gamma}{1 - \alpha} + \sqrt{ \left( \frac{4\beta \gamma}{1 - \alpha} \right)^2 + 4 \frac{2\beta \kappa + 4\beta \gamma}{1 - \alpha} } }{2} \right]^k,
    \end{equation}
    where
    \begin{equation}\label{parameters}
        \begin{aligned}
            \gamma &= \lvert A \rvert^2 + \lvert B \rvert^2, \\
            \kappa &= \left( \lvert a \rvert^2 + \lvert b \rvert^2 \right)^2, \\
            \alpha &= \lvert A \rvert^2 + \lvert B \rvert^2 + \lvert A \overline{a}^2 - \lvert B \rvert^2 - \lvert A \rvert^2 + \sqrt{2} B \overline{ab} \rvert, \\
            \beta &= -2 \mathrm{Re} \left( \overline{A} a^2 \right) + \lvert A \rvert^2 + \lvert B \rvert^2 - 2\sqrt{2} \mathrm{Re} \left( \overline{B} ab \right) \\
            &\quad + \left( \lvert a \rvert^2 + \lvert b \rvert^2 \right)^2 + \lvert A \overline{a}^2 - \lvert B \rvert^2 - \lvert A \rvert^2 + \sqrt{2} B \overline{ab} \rvert.
        \end{aligned}
    \end{equation}
\end{Coro}

\begin{proof}
    Applying sampling rule 1 to 
    $\mathbb{E}\left[\lvert\alpha ^{k}_n-u_{n}\rvert^2\right]$in equation \eqref{3.15} yields
    \begin{eqnarray}\label{3.27}
        \begin{aligned}
            \mathbb{E}\left[\lvert\alpha ^{k}_n-u_{n}\rvert^2\right]&=\mathbb{E}\left[\lvert \mathcal{F}_{\Delta t}(U^{k-1}_{n-1})+(\sigma^k_n \circ z^k_n)-u_n\rvert^2\right]\\
            &\le 2\mathbb{E}\left[\lvert \mathcal{F}_{\Delta t}(U^{k-1}_{n-1})-u_n\rvert^2\right]+2\mathbb{E}\left[\lvert\sigma^k_n \circ z^k_n\rvert^2\right]\\
            &\le 2\mathbb{E}\left[\lvert\mathcal{F}_{\Delta t}(U^{k-1}_{n-1})-u_n\rvert^2\right]+2\mathbb{E}\left[\lvert\sigma^k_n\rvert^2\right].
        \end{aligned}
    \end{eqnarray}
    In equation \eqref{3.27}, the term $\mathbb{E}\left[\lvert\sigma^k_n\rvert^2\right]$ has already been computed in equation \eqref{3.17}. We now proceed to analyze $\mathbb{E}\left[\lvert \mathcal{F}_{\Delta t}(U^{k-1}_{n-1})-u_n\rvert^2\right]$
    \begin{eqnarray}\label{3.28}
        \begin{aligned}
            \mathbb{E}\left[\lvert \mathcal{F}_{\Delta t}(U^{k-1}_{n-1})-u_n\rvert^2\right]
            &\le\mathbb{E}\left[\lvert(a+b\upsilon_{n+\frac{1}{2}})(a+b\upsilon_n)(U^{k-1}_{n-1}-u_{n-1})\rvert^2\right]\\
            &\le
            \kappa \mathbb{E}\left[\lvert U^{k-1}_{n-1}-u_{n-1}\rvert^2\right].
        \end{aligned}
    \end{eqnarray}
     where $\kappa =\left(\lvert a\rvert^2+\lvert b\rvert^2\right)^2 $. Substituting equations \eqref{3.28} and \eqref{3.17} into equation \eqref{3.15} results in
    \begin{align}\label{3.30}
        \mathbb{E}\left[\lvert U_{n+1}^{k+1}-u_{n+1}\rvert^2\right]&\le \alpha\mathbb{E}\left[\lvert U_{n}^{k+1}-u_{n}\rvert^2\right]+4\beta\gamma\mathbb{E}\left[\lvert U^k_{n-1}-u_{n-1}\rvert^2\right] \nonumber\\
        &+(2\beta\kappa+4\beta\gamma)\mathbb{E}\left[\lvert U^{k-1}_{n-1}-u_{n-1}\rvert^2\right].
    \end{align}
Let \(\epsilon_n^k = \mathbb{E}\left[ \lvert U_{n+1}^{k+1} - u_{n+1} \rvert^2 \right]\). Taking the supremum over \(n\), we obtain the following error bound
\begin{equation}\label{3.31}
    \sup_{n} \epsilon_{n+1}^{k+1} \leq \alpha \sup_{n} \epsilon_{n}^{k+1} + 4\beta\gamma \sup_{n} \epsilon_{n-1}^{k} + (2\beta\kappa + 4\beta\gamma) \sup_{n} \epsilon_{n-1}^{k-1}.
\end{equation}
Let \(\widehat{\epsilon}^{(k)} = \sup\limits_{n} \epsilon_n^k\). From equation \eqref{3.31}, we derive the following error bound
\begin{equation}\label{3.32}
    \widehat{\epsilon}^{(k+1)} \leq \alpha \widehat{\epsilon}^{(k+1)} + 4\beta\gamma \widehat{\epsilon}^{(k)} + (2\beta\kappa + 4\beta\gamma) \widehat{\epsilon}^{(k-1)}.
\end{equation}
Simplifying further, we obtain
\begin{equation}
    \widehat{\epsilon}^{(k+1)} \leq \frac{4\beta\gamma}{1-\alpha} \widehat{\epsilon}^{(k)} + \frac{2\beta\kappa + 4\beta\gamma}{1-\alpha} \widehat{\epsilon}^{(k-1)}.
\end{equation}
It is easy to verify that $\frac{4\beta\gamma}{1-\alpha}$ and $\frac{2\beta\kappa + 4\beta\gamma}{1-\alpha}$ are nonnegative, thus from Lemma \ref{lemma}, we have
\begin{equation}\label{3.33}
    \widehat{\epsilon}^{(k)} \leq \widehat{\epsilon}^{(0)} \left[ \frac{\frac{4\beta\gamma}{1-\alpha} + \sqrt{\left( \frac{4\beta\gamma}{1-\alpha} \right)^2 + 4\frac{2\beta\kappa + 4\beta\gamma}{1-\alpha}}}{2} \right]^k.
\end{equation}   
\end{proof}

\begin{Rem}
    The proof for sampling rule 3 follows the same reasoning as for sampling rule 1, with the only difference being that \(\mathbb{E}\left[\lvert \sqrt{3}(2\omega^k_n - 1)\rvert^2\right] = 1\) is used in place of \(\mathbb{E}\left[\lvert z^k_n \rvert^2\right] = 1\).
\end{Rem}


\section{Numerical Experiments}

In this section, we use the SParareal algorithm to solve stochastic differential equations (SDEs). The parameters are set as $N = N_g$ and $M = \frac{N_f}{N_g} = \frac{\Delta T}{\Delta t} = 2$. Unless otherwise stated, the numerical error comparison plots for $P$ and $P_s$ are generated based on sampling rule 1. All numerical experiments are conducted with $m=2$ for 5 independent simulations.

The convergence criterion for the SParareal algorithm is defined as:
\begin{equation}\label{stoping_criterion_epsilon}
    \lVert U^k_n - U^{k-1}_n \rVert_{\infty} \leq \varepsilon, \quad 0 \leq n \leq N,
\end{equation}
where $\varepsilon = 10^{-12}$ is used in this section.

During the numerical experiments, we evaluate the efficiency of the algorithm by comparing the number of iterations $k$ required to achieve a given error threshold $\rho$. Specifically, $k$ is the smallest integer satisfying the following condition:
\begin{equation}
    \mathbb{E}\lVert U^k_n - u_n \rVert_{\infty} \leq \rho, \quad 0 \leq n \leq N,
\end{equation}
where $\rho = 10^{-12}$ is used in this section.

\subsection{Linear Stochastic Differential Equation}
\begin{example}
In this section, we consider the Dahlquist test equation for numerical experiments:
\begin{equation}\label{eq:dahlquist}
    du(t) = \lambda u(t) dt + \mu u(t) dW(t), \quad t \in [0, T],
\end{equation}
with the initial condition $u(0) = 1$.
\end{example}



\subsection*{Case 1: SDEs with Real Coefficient}

To evaluate the performance of the SParareal algorithm ($P_s$) and compare it with the deterministic Parareal algorithm ($P$), we conduct numerical experiments under the following setup. The time interval $T \in [0,3]$ is discretized into $N=40$ sub-intervals, with a coarse time step $\Delta T = \frac{3}{40}$ and a fine time step $\Delta t = \frac{3}{80}$. The parameters are set as $\theta_g = 1$, $\theta_f = 0.5$, $\lambda = -40$, and $\mu = 0.56$. 

Figure~\ref{fig:subfig1} presents the numerical solutions computed by $P_s$ and $\mathcal{F}_{\Delta t}$, showing that their values are in close agreement. The error comparison in Figure~\ref{fig:subfig2} demonstrates that $P$ converges in $k=11$ iterations to meet the error tolerance $\varepsilon$, while $P_s$, using only two samples, converges in $k=8$ iterations. 

Figure~\ref{fig:main1} compares the maximum theoretical error bound with the numerical errors of $P_s$, plotted against the iteration number $k$, under the condition $\alpha < 1$ as established in Theorem~\ref{Thm2.1}. 
Additionally, the numerical errors for sampling rules 2 and 4, as well as for sampling rules 1 and 3, are nearly indistinguishable, respectively.


Figure~\ref{fig:mainu} illustrates the impact of varying the number of samples $m$ on the convergence behavior of the algorithm. As $m$ increases from 7 to 125, the required number of iterations $k$ decreases correspondingly. Specifically, for $m = 7$, $k$ reduces to 6; for $m = 20$, $k$ decreases to 5; and for $m = 125$, $k$ further reduces to 4. However, beyond this point, increasing $m$ to 1000 does not lead to a further reduction in $k$ but instead results in a gradual decrease in numerical errors. This observation underscores the diminishing returns of increasing $m$ and highlights the importance of balancing $m$ to optimize computational cost,  as discussed in \cite{2021Traffic}.


\begin{figure}[htbp]
    \subfigure[Numerical Solutions]{
        \includegraphics[width=0.45\textwidth]{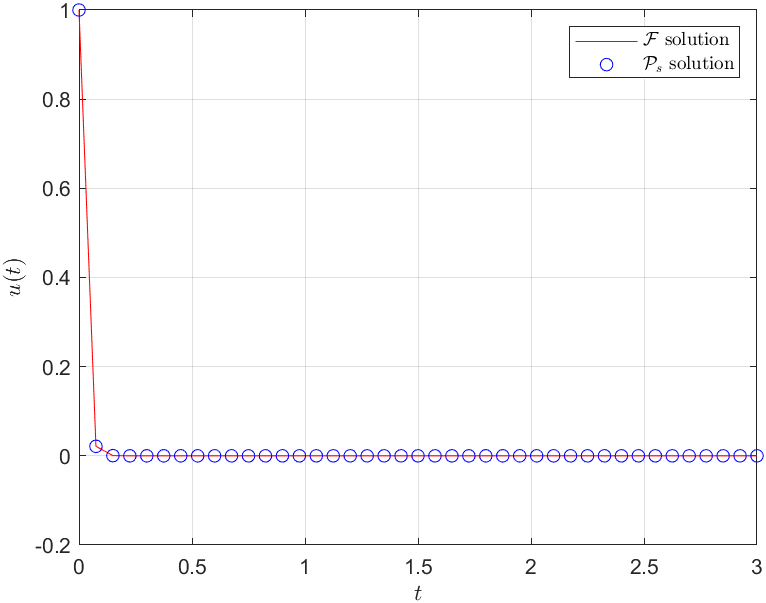}
        \label{fig:subfig1}
    }
    \subfigure[Numerical Errors]{
        \includegraphics[width=0.45\textwidth]{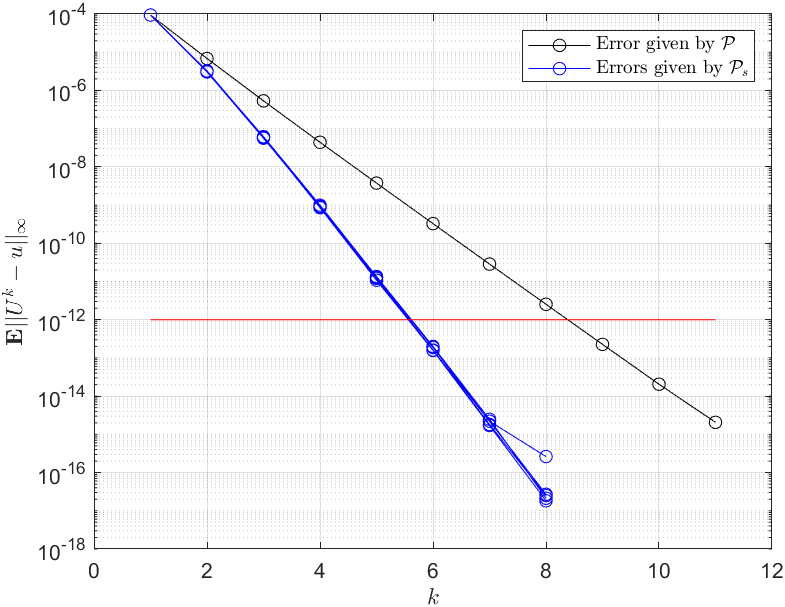}
        \label{fig:subfig2}
    }
    \captionsetup{font={footnotesize}}
    \caption{
    \subref{fig:subfig1} Numerical solutions obtained by $\mathcal{F}_{\Delta t}$ and $P_s$. Only one time-independent simulation is shown for clarity. \subref{fig:subfig2} Numerical errors for $P$ (black line) and $P_s$ (blue line). The horizontal red dashed line indicates the error threshold $\rho = 10^{-12}$.
    }
\end{figure}

\begin{figure}[htbp]
    \subfigure[Sampling Rules 1 and 3]{
        \includegraphics[width=0.45\textwidth]{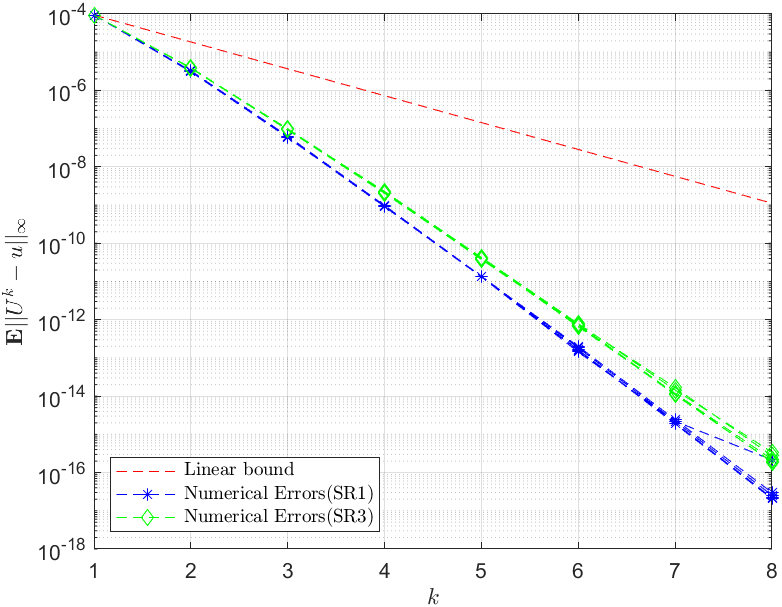}
        \label{fig:subfig3}
    }
    \subfigure[Sampling Rules 2 and 4]{
        \includegraphics[width=0.45\textwidth]{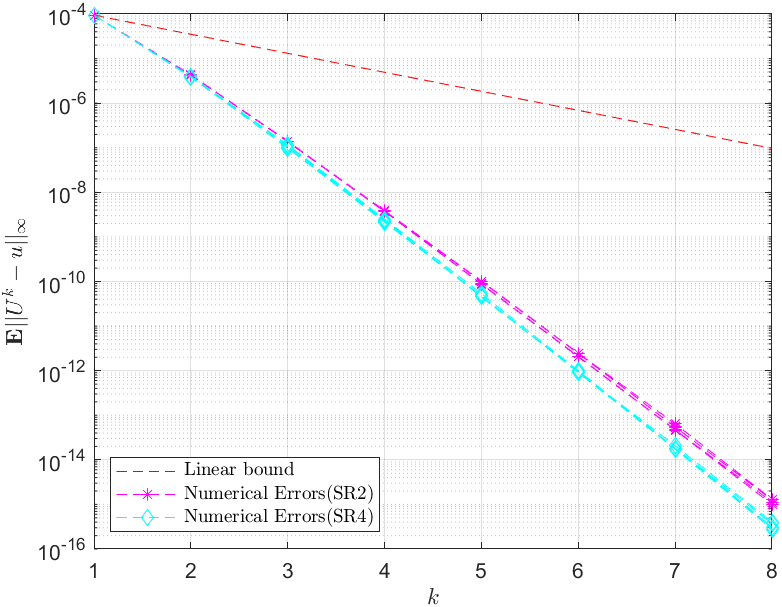}
        \label{fig:subfig4}
    }
    \captionsetup{font={footnotesize}}
    \caption{
    Linear error bounds from Corollary~\ref{Corollary} (red line) compared with numerical errors.
    \subref{fig:subfig3} Errors for sampling rules 1 and 3 (blue and green lines). 
    \subref{fig:subfig4} Errors for sampling rules 2 and 4 (purple and cyan lines). 
    }
    \label{fig:main1}
\end{figure}

\begin{figure}[htbp]
    \centering
    \subfigure[$m=2$]{
        \includegraphics[width=0.3\textwidth]{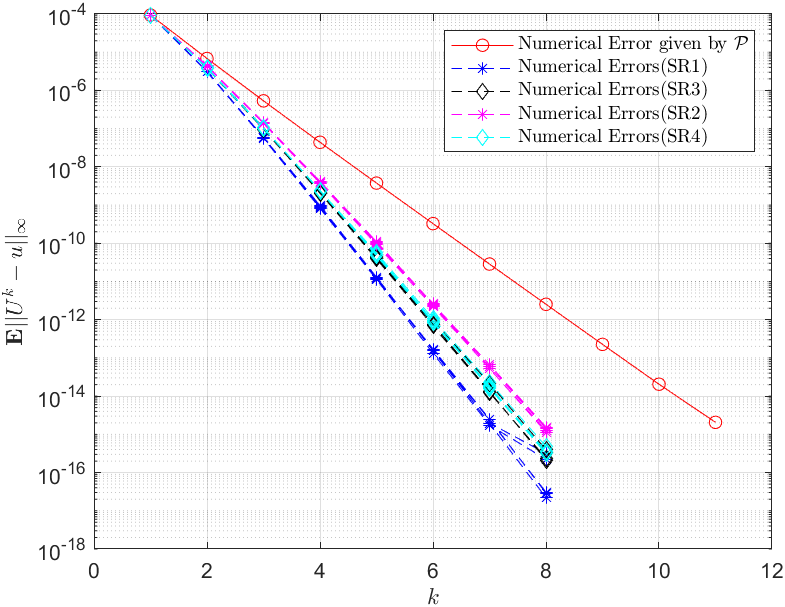}
    }
    \subfigure[$m=7$]{
        \includegraphics[width=0.3\textwidth]{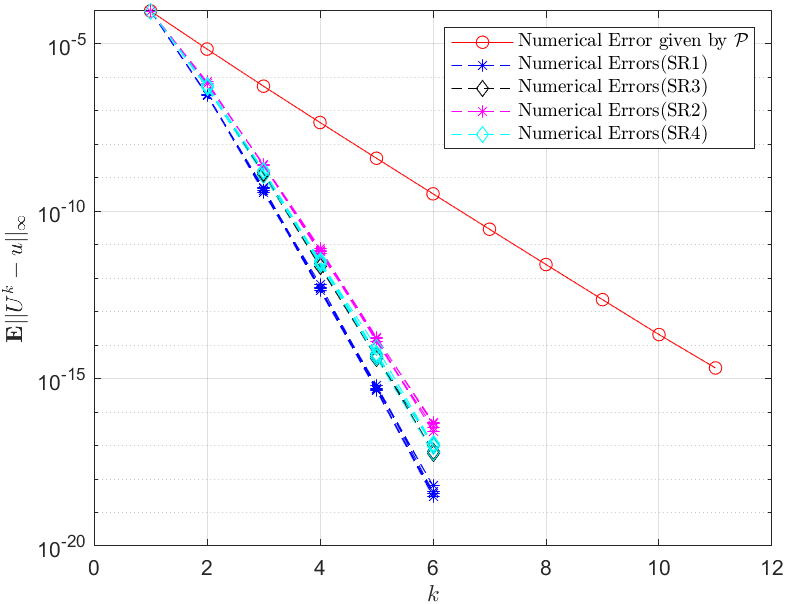}
    }
    \subfigure[$m=20$]{
        \includegraphics[width=0.3\textwidth]{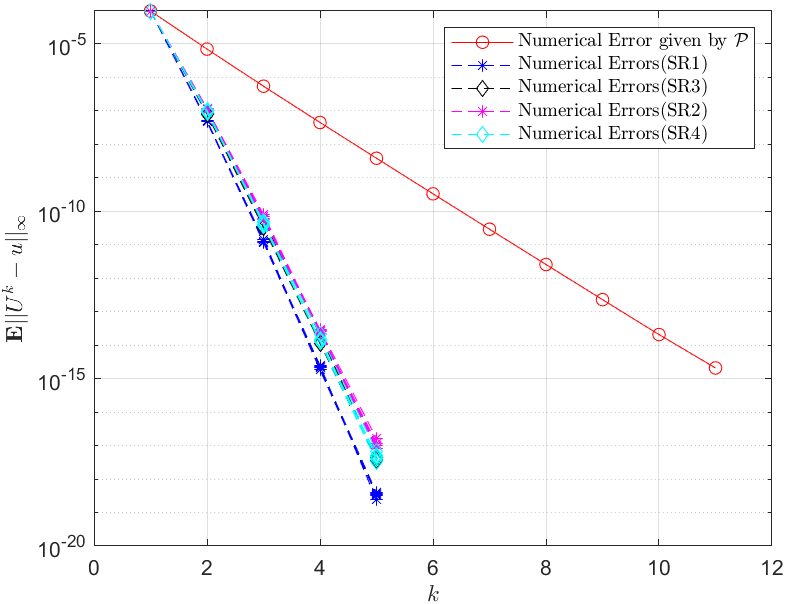}
    }
    
    \subfigure[$m=125$]{
        \includegraphics[width=0.3\textwidth]{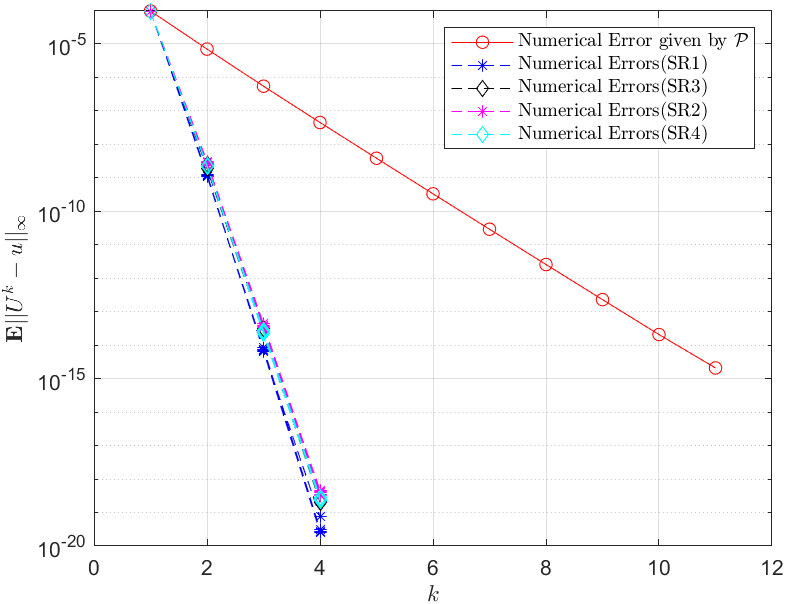}
    }
    \subfigure[$m=500$]{
        \includegraphics[width=0.3\textwidth]{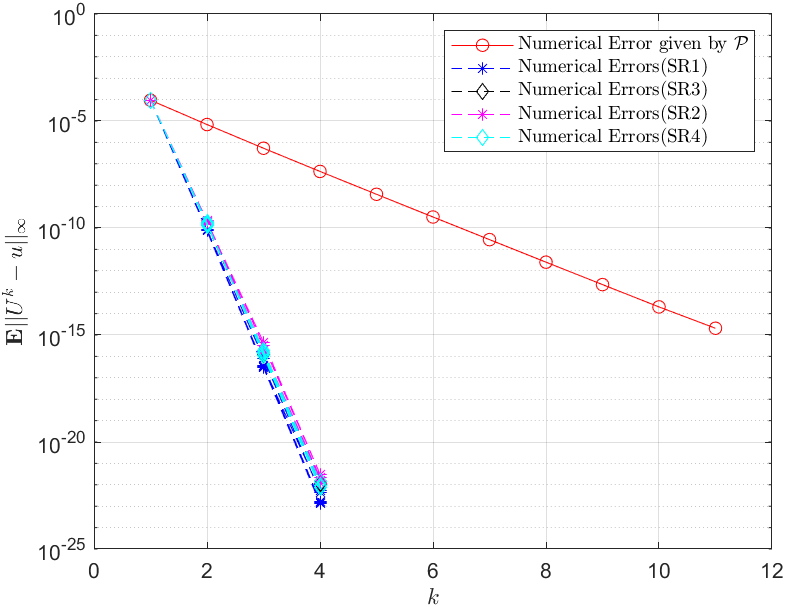}
    }
    \subfigure[$m=1000$]{
        \includegraphics[width=0.3\textwidth]{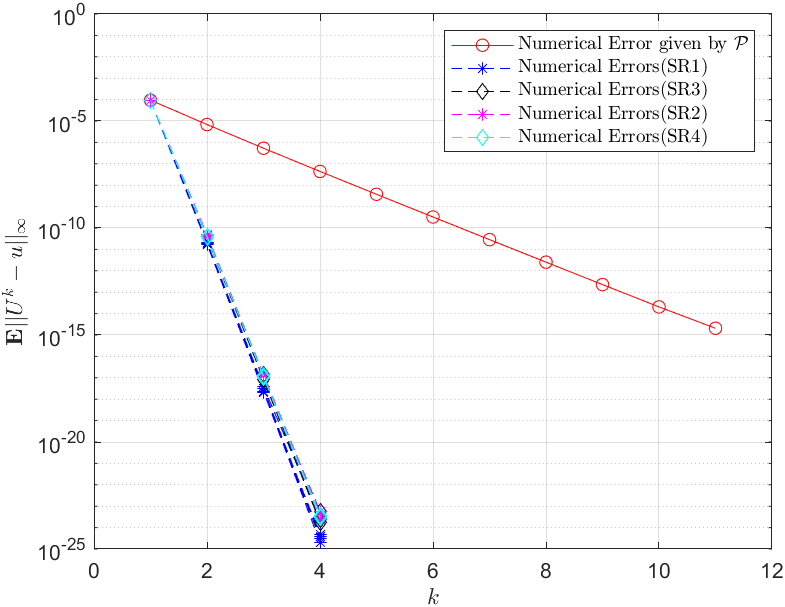}
    }
    \captionsetup{font={footnotesize}}
    \caption{
    Numerical errors for $P_s$ with varying sample sizes $m$. Sampling rules 1/3 are represented in blue and black, while sampling rules 2/4 are shown in purple and cyan. From left to right, $m$ takes values 2, 7, 20, 125, 500, and 1000.
    }
    \label{fig:mainu}
\end{figure}

\subsection*{Case 2: SDEs with Extended Time Interval }

In this experiment, we evaluate the convergence behavior of the SParareal algorithm ($P_s$) and the deterministic Parareal algorithm ($P$) over an extended time interval $T \in [0,9]$. The time interval is discretized into $N=40$ sub-intervals, with a coarse time step   $\Delta T = \frac{9}{40}$ and a fine time step $\Delta t = \frac{9}{80}$. The parameters are set as $\theta_g = 0.5$, $\theta_f = 1$, $\lambda = -40$, and $\mu = 0.56$.

Figure~\ref{fig:main3} illustrates the numerical errors and convergence behavior under these parameters. For $P_s$, all sampling rules converge in $k=27$ iterations with $m=2$. Increasing the number of samples  $m$ improves the convergence rate, highlighting the effectiveness of $P_s$ under larger $m$. In contrast, $P$ requires $k=39$ iterations to meet the stopping criterion \eqref{stoping_criterion_epsilon}.  Moreover, the linear error bound (see \eqref{3.7} and \eqref{3.26}) for $P_s$ does not converge under these extended parameters, as shown by the red dotted line in Figure~\ref{fig:main3}.




\begin{figure}[htbp]
    \subfigure[Numerical Error]{
        \includegraphics[width=0.3\textwidth]{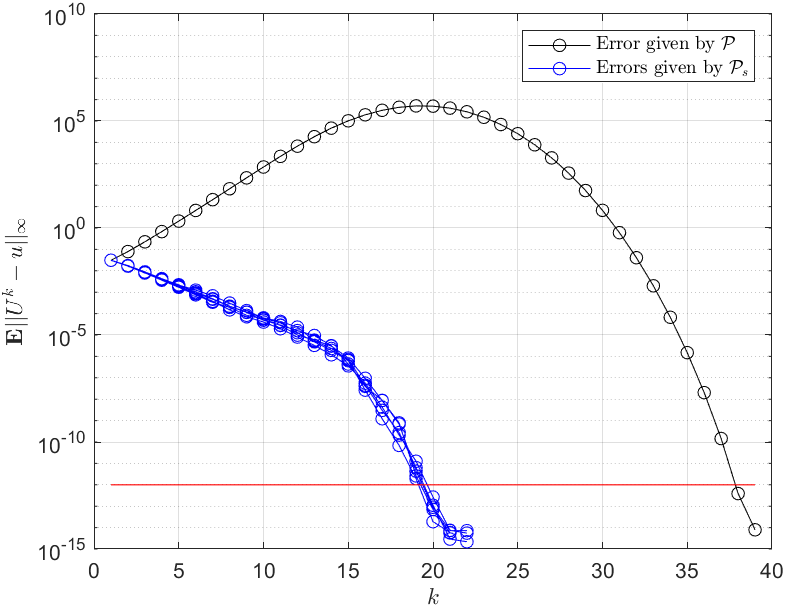}
        \label{fig:subfig6}
    }
    \subfigure[Sampling Rules 1 and 3]{
        \includegraphics[width=0.3\textwidth]{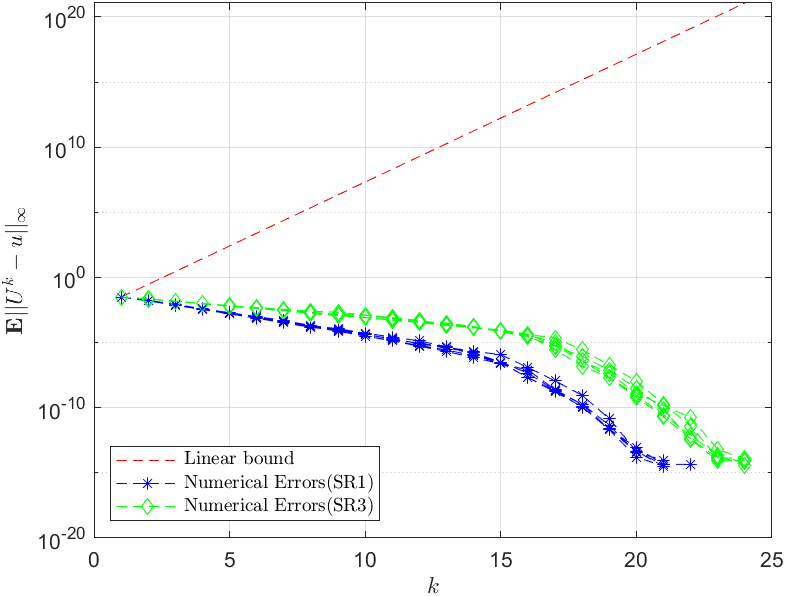}
        \label{fig:subfig7}
    }
    \subfigure[Sampling Rules 2 and 4]{
        \includegraphics[width=0.3\textwidth]{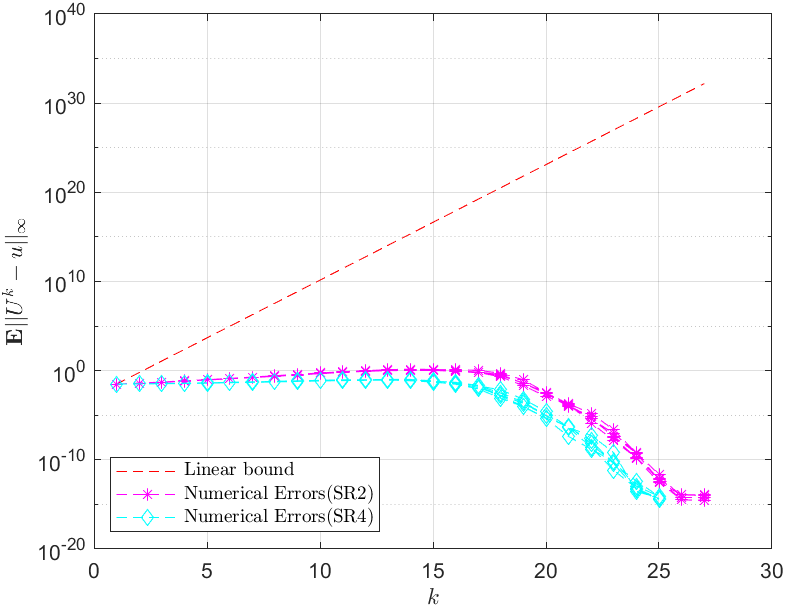}
        \label{fig:subfig8}
    }
    \captionsetup{font={footnotesize}}
    \caption{
    Numerical errors and convergence behavior for $P_s$ and $P$ with $T \in [0,9]$, $\Delta T = \frac{9}{40}$, and $\Delta t = \frac{9}{80}$. 
    \subref{fig:subfig6}  Comparsion of $P$ and $P_s$. 
    \subref{fig:subfig7} Numerical errors for sampling rules 1 and 3 (blue and green lines). 
    \subref{fig:subfig8} Numerical errors for sampling rules 2 and 4 (purple and cyan lines).}
    \label{fig:main3}
\end{figure}

\subsection*{Case 3: SDEs with Imaginary Coefficient}
In this experiment, the time interval $T \in [0,3]$ is discretized into $N=40$ sub-intervals, with a coarse time step $\Delta T = \frac{3}{40}$ and a fine time step $\Delta t = \frac{3}{80}$. The parameters are set as $\theta_g=1$, $\theta_f=0.5$, $\lambda=-40$, and $\mu=0.56 + i$.


As can be clearly observed in Figure~\ref{fig:subfig10}, the convergence rate of the stochastic  parareal algorithm 
$P_s$ is faster than that of the standard parareal algorithm $P$. And
 for sampling rules 1 and 3, Rule 1 exhibits an advantage over Rule 3 when $\mu$
 is imaginary. In contrast, for sampling rules 2 and 4, their convergence speed are comparable.

\begin{figure}[htbp]
    \subfigure[Numerical Errors]{
        \includegraphics[width=0.3\textwidth]{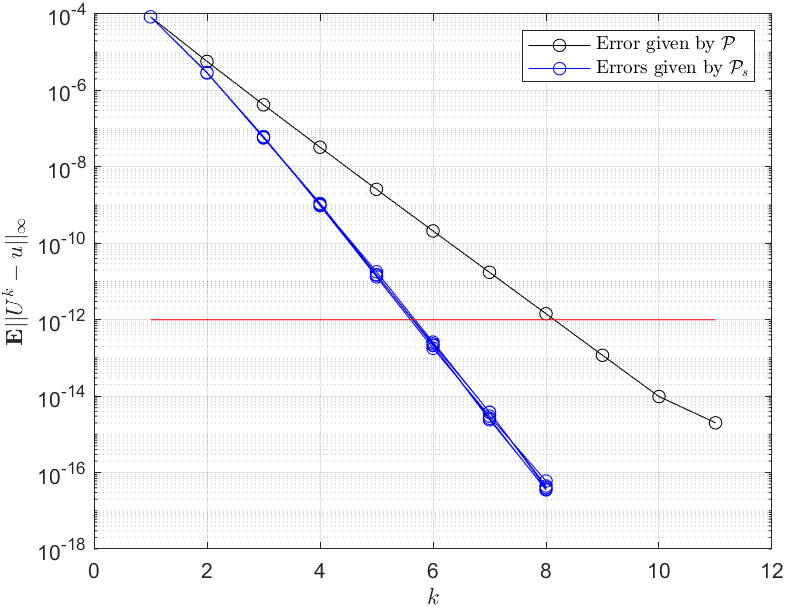}
        \label{fig:subfig10}
    }
    \subfigure[Sampling Rules 1 and 3]{
        \includegraphics[width=0.3\textwidth]{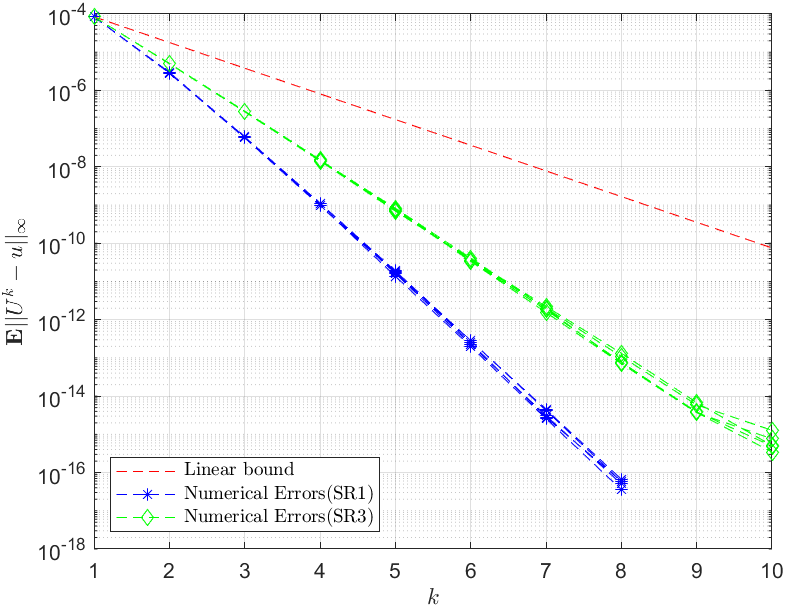}
        \label{fig:subfig11}
    }
    \subfigure[Sampling Rules 2 and 4]{
        \includegraphics[width=0.3\textwidth]{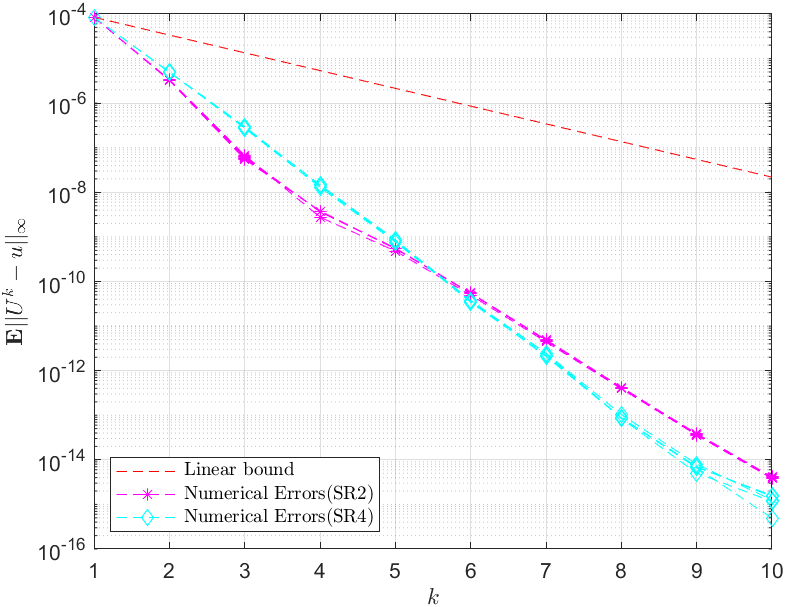}
        \label{fig:subfig12}
    }
    \captionsetup{font={footnotesize}}
    \caption{
    Numerical errors for $P_s$ with $\mu = 0.56 + i$ (imaginary $\mu$). 
    \subref{fig:subfig10} Comparsion of $P$ and $P_s$.  \subref{fig:subfig11} Comparison of sampling rules 1 and 3. 
    \subref{fig:subfig12} Comparison of sampling rules 2 and 4. }
    \label{fig:main4}
\end{figure}

\subsection{Nonlinear SDEs}
\begin{example}
\textbf{Phase Transition Model} \par
The Ginzburg-Landau equation, widely used to describe phase transition phenomena in superconductivity \cite{Ginzburg2009}. The stochastic variant of this model is formulated as:
\begin{eqnarray}\label{Phase_Transition_eq}
    dX(t) = \left((\upsilon + \frac{1}{2}\sigma^2) X(t) - \lambda X(t)^3 \right) dt + \sigma X(t) dW(t),
\end{eqnarray}
where the parameters satisfy $\upsilon \geq 0$, $\lambda > 0$, and $\sigma > 0$. For the initial condition $X(0) > 0$, the exact solution can be expressed as:
\begin{eqnarray}
    X(t) = \frac{X(0) e^{\upsilon t + \sigma W(t)}}{\sqrt{1 + 2 X(0)^2 \lambda \int_0^t e^{2\upsilon s + 2\sigma W(s)} ds}}.
\end{eqnarray}\par
\end{example}
For the numerical experiments, we set the initial condition $X(0) = 1$, with the time interval $T \in [0, 1]$, coarse time step $\Delta T = \frac{1}{40}$, and fine time step $\Delta t = \frac{1}{80}$. The remaining parameter values are chosen as follows: $\theta_g = 0$, $\theta_f = 0$, $\upsilon = 25$, $\sigma = 0.5$, and $\lambda = 0.1$.

The numerical solution of the stochastic phase transition model $P_s$ is shown in Figure  \ref{Nolinear/Phase Transition Model1}. In Figure \ref{Nolinear/Phase Transition Model2}, it is observed that the method $P$ converges to a solution within $k = 22$ iterations under the specified error tolerance $\varepsilon$, while the method $P_s$ converges in $k = 16$ iterations under the same tolerance.

We further analyze the relationship between the number of iterations $k$ and the number of samples $m$ used for the stochastic phase transition model \eqref{Phase_Transition_eq}. From Figure \ref{fig:main Nolinear/Phase Transition Model}, we observe that as $m$ increases, the required number of iterations $k$ decreases correspondingly. Moreover, when $m = 100$, further increases in $m$ do not result in noticeable changes in the number of iterations $k$, but instead lead to a reduction in numerical errors.

    \begin{figure}[htbp]
     \subfigure[Numerical solutions]{
        \includegraphics[width=0.3\textwidth]{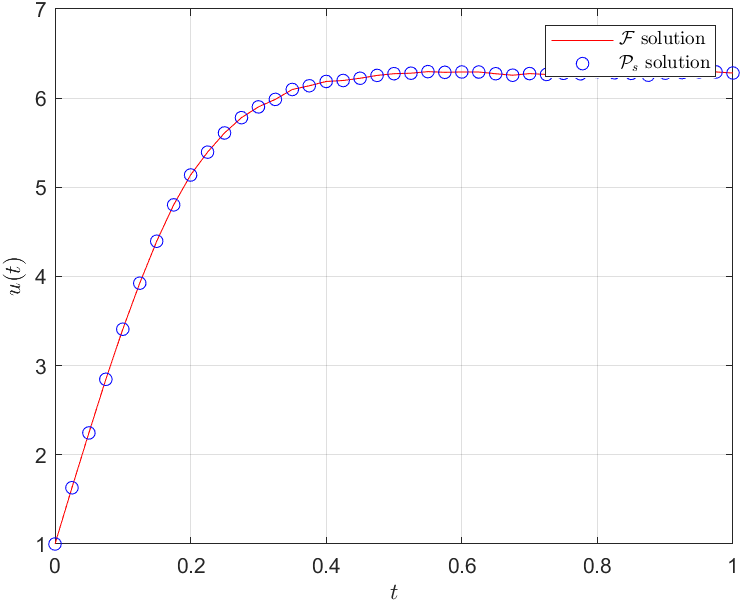}
        \label{Nolinear/Phase Transition Model1}
    }
        \subfigure[Sampling Rules 1 and 3 ]{
        \includegraphics[width=0.3\textwidth]{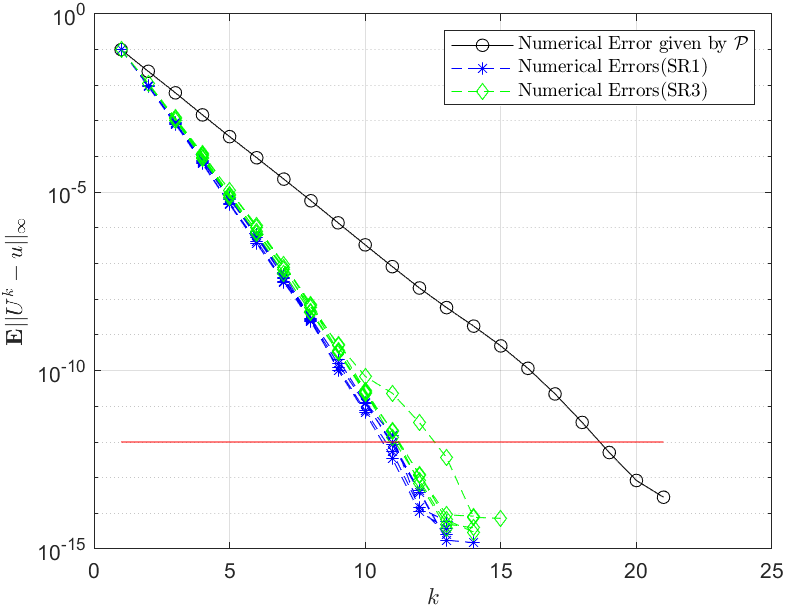}
        \label{Nolinear/Phase Transition Model2}
    }
         \subfigure[Sampling Rules 2 and 4]{
         \includegraphics[width=0.3\textwidth]{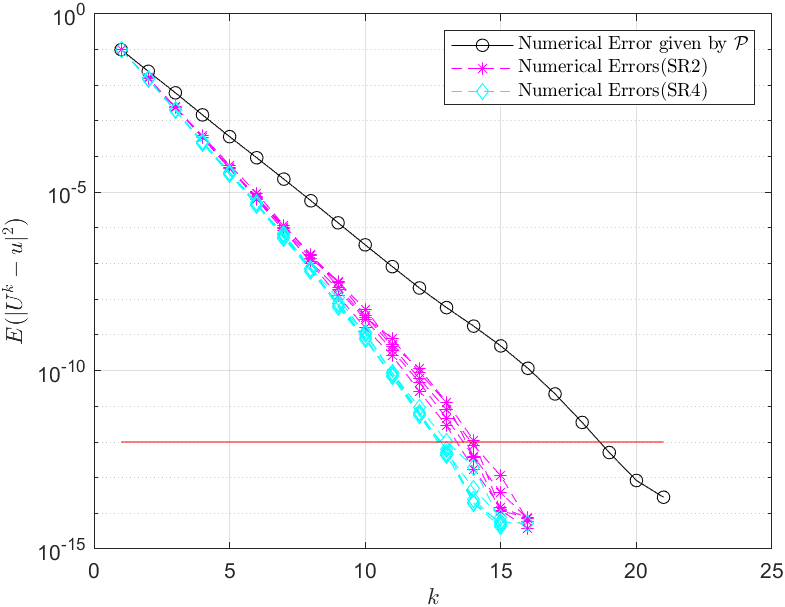}
        \label{Nolinear/Phase Transition Model3}
    }
    \captionsetup{font={footnotesize}}
    \caption{
    \subref{Nolinear/Phase Transition Model1} Numerical solutions obtained by $\mathcal{F}_{\Delta t}$ and $P_s$. \subref{Nolinear/Phase Transition Model2} Comparsion of $P$, $P_s$ with sampling rules 1 and 3. 
    \subref{Nolinear/Phase Transition Model3} Comparsion of $P$, $P_s$ with sampling rules 2 and 4. The horizontal red dashed line indicates the error threshold $\rho = 10^{-12}$.
    }
    \label{Nolinear/Phase Transition Model}
    \end{figure}
    
    \begin{figure}[htbp]
     \subfigure[m=2]{
         \includegraphics[width=0.3\textwidth]{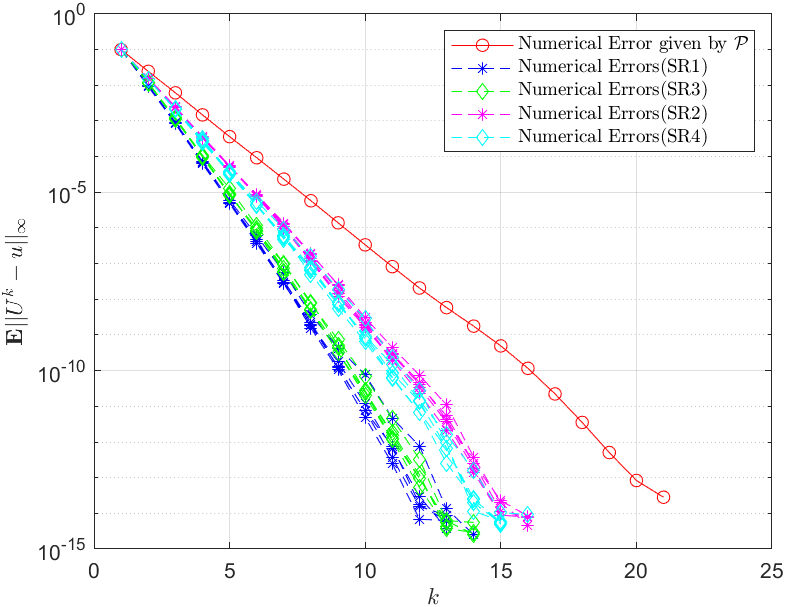}
         \label{Nolinear/Phase Transition Modelm=2}
    }
        \subfigure[m=4]{
        \includegraphics[width=0.3\textwidth]{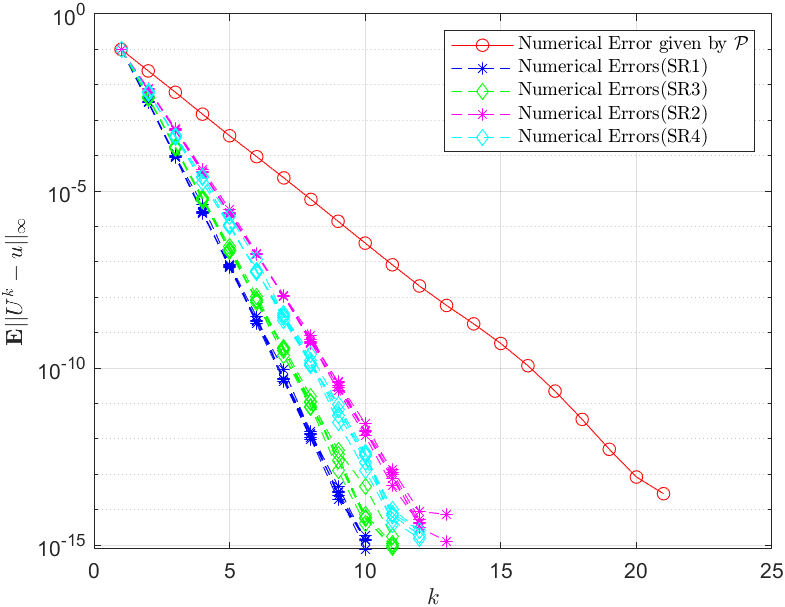}
        \label{Nolinear/Phase Transition Modelm=4}
    }
     \subfigure[m=9]{
         \includegraphics[width=0.3\textwidth]{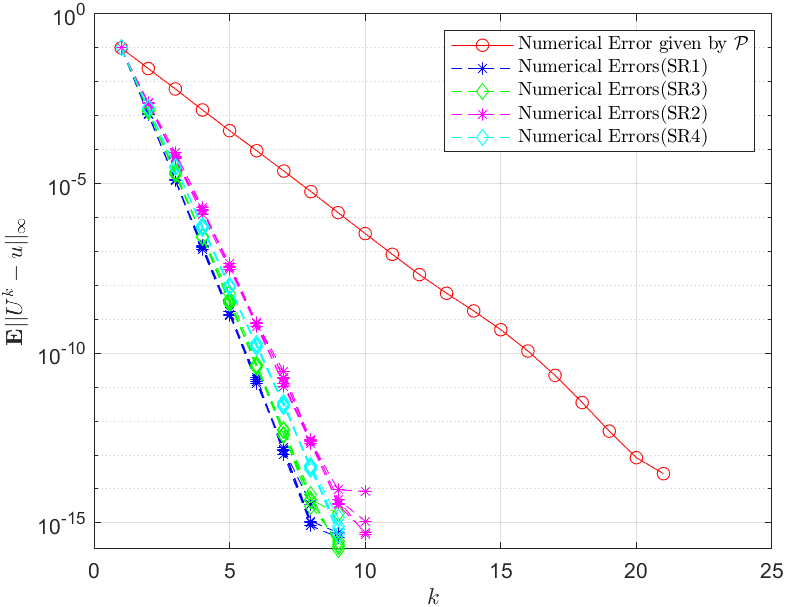}
         \label{Nolinear/Phase Transition Modelm=9}
    }
    
     \subfigure[m=20]{
        \includegraphics[width=0.3\textwidth]{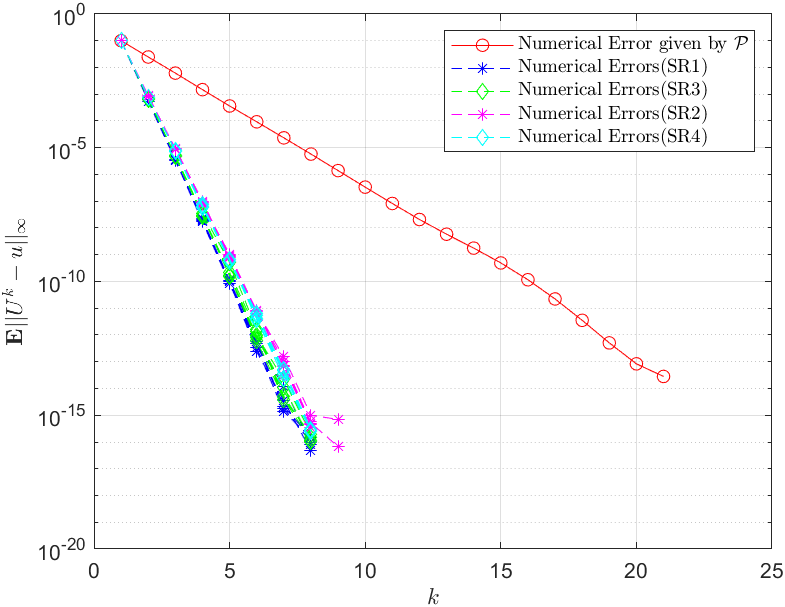}
        \label{Nolinear/Phase Transition Modelm=20}
    }
         \subfigure[m=100]{
        \includegraphics[width=0.3\textwidth]{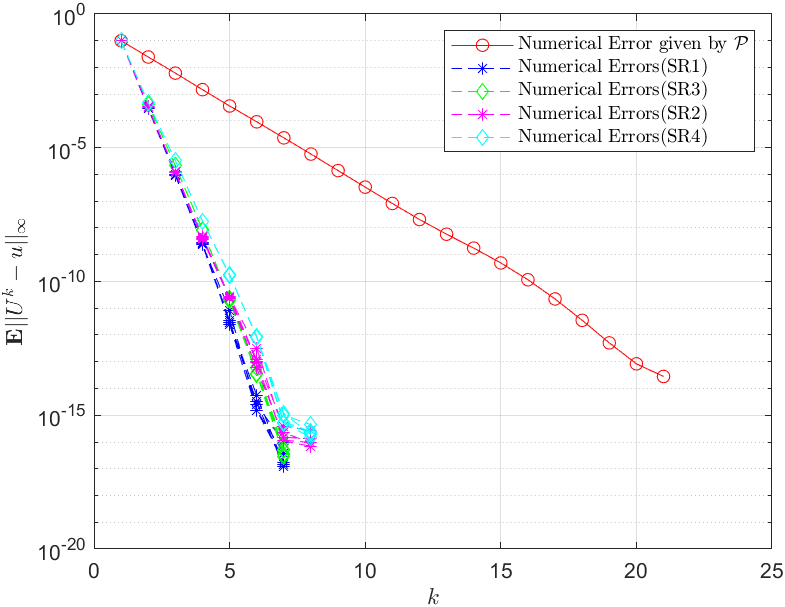}
         \label{Nolinear/Phase Transition Modelm=100}
    }
    \subfigure[m=1000]{
         \includegraphics[width=0.3\textwidth]{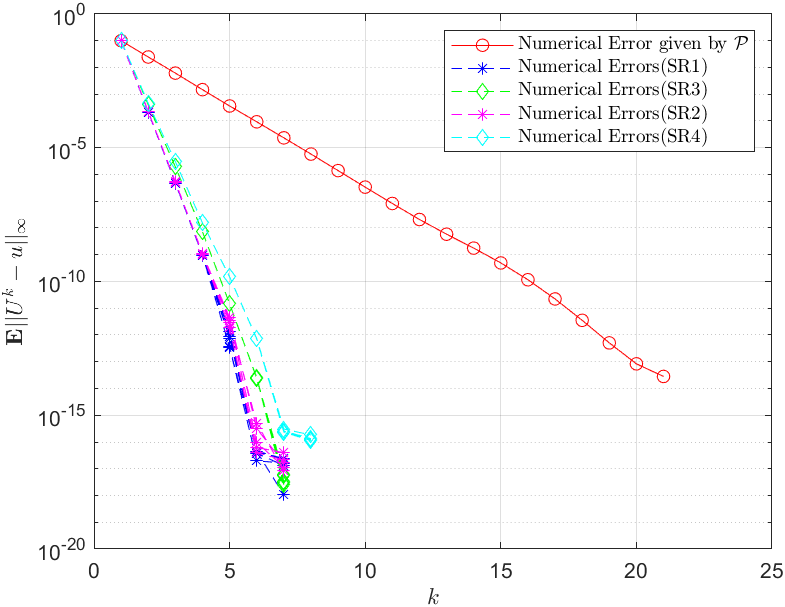}
         \label{Nolinear/Phase Transition Modelm=1000}
    }
    \captionsetup{font={footnotesize}}
    \caption{
    Numerical errors of $P_s$ for varying sample sizes $m$. Results for sampling rules 1 and 3 are represented by blue and green lines, respectively, while sampling rules 2 and 4 are depicted in purple and cyan. From left to right, the number of samples $m$ are set to 2, 4, 9, 11, 20, and 1000.
    }
    \label{fig:main Nolinear/Phase Transition Model}
    \end{figure}

\begin{example}
    \textbf{Double-Well Potential} \cite{doi:10.1137/050638667} \par
    The ordinary differential equation (ODE) $\frac{dx}{dt} = -V'(x(t))$, where the function $V$ is bounded below, is referred to as a potential. By applying the chain rule, we derive the following relationship:
    \begin{eqnarray}
        \frac{d}{dt} V(x(t)) = V'(x(t)) \frac{dx(t)}{dt} = -(V'(x(t)))^2.
    \end{eqnarray}
    This indicates that along any solution trajectory, the potential energy $V(x(t))$ is non-increasing. Moreover, $V(x(t))$ is strictly decreasing unless the solution reaches a stationary point. For the case of a double-well potential, we define:
    \begin{eqnarray}\label{4.6}
        V(x) = x^2(x - 2)^2.
    \end{eqnarray}

    The behavior of the solution depends on the initial condition. For $x(0) < 0$, the numerical solution of the ODE monotonically decreases to $0$, sliding along the left branch of the left well. Similarly:
    \begin{itemize}
        \item For $0 < x(0) < 1$, the solution descends along the right branch of the left well.
        \item For $1 < x(0) < 2$, the solution moves down the left branch of the right well.
        \item For $x(0) > 2$, the solution slides down the right branch of the right well.
    \end{itemize}
    When noise is introduced into the problem, the solution can overcome the potential barrier and cross the central peak, transitioning from one well to another. The general form of the stochastic differential equation with additive noise is given by:
    \begin{eqnarray}
        dX(t) = -V'(X(t)) dt + \sigma dW(t),
    \end{eqnarray}
    where $\sigma$ is a constant. For the double-well potential defined in equation \eqref{4.6}, this becomes:
    \begin{eqnarray}\label{4.9}
        dX(t) = (-8X(t) + 12X(t)^2 - 4X(t)^3) dt + \sigma dW(t).
    \end{eqnarray}
\end{example}

In the numerical experiments, we set the initial condition $u(0) = 1$, with $N = 40$, and consider the time interval $T \in [0, 1]$. The coarse time step is chosen as $\Delta T = \frac{1}{40}$, and the fine time step is $\Delta t = \frac{1}{80}$. The remaining parameters are set as $\theta_g = 0$, $\theta_f = 0$, and $\sigma = 4$.

Figure \ref{Nolinear/Double-Well Potential1} demonstrates that the numerical solutions computed by $P_s$ and $\mathcal{F}_{\Delta t}$ are in close agreement. Furthermore, Figures \ref{Nolinear/Double-Well Potential2} and \ref{Nolinear/Double-Well Potential3} indicate that the advantages of $P_s$ over $P$ are not evident under the given conditions. For the same number of iterations, the numerical errors of $P_s$ and $P$ are very similar.

    \begin{figure}[htbp]
    \subfigure[Numerical solutions]{
        \includegraphics[width=0.3\textwidth]{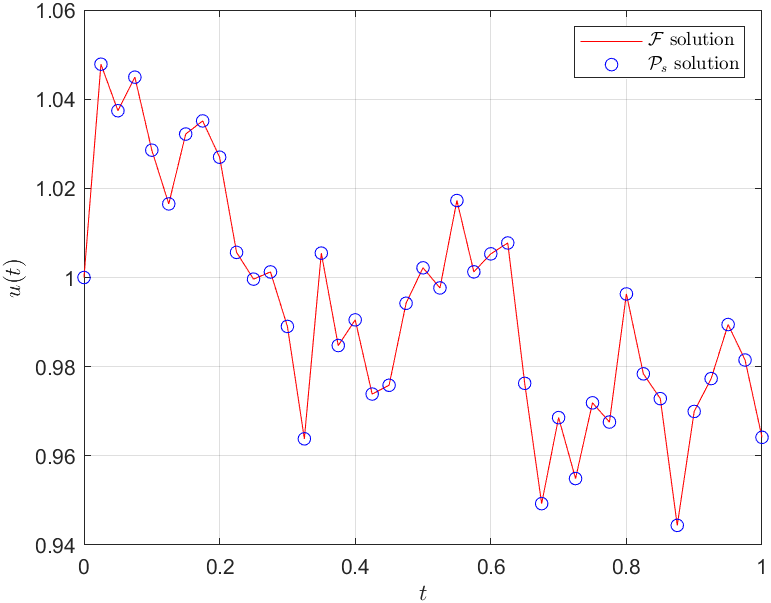}
        \label{Nolinear/Double-Well Potential1}
    }
        \subfigure[Sampling Rules 1 and 3]{
        \includegraphics[width=0.3\textwidth]{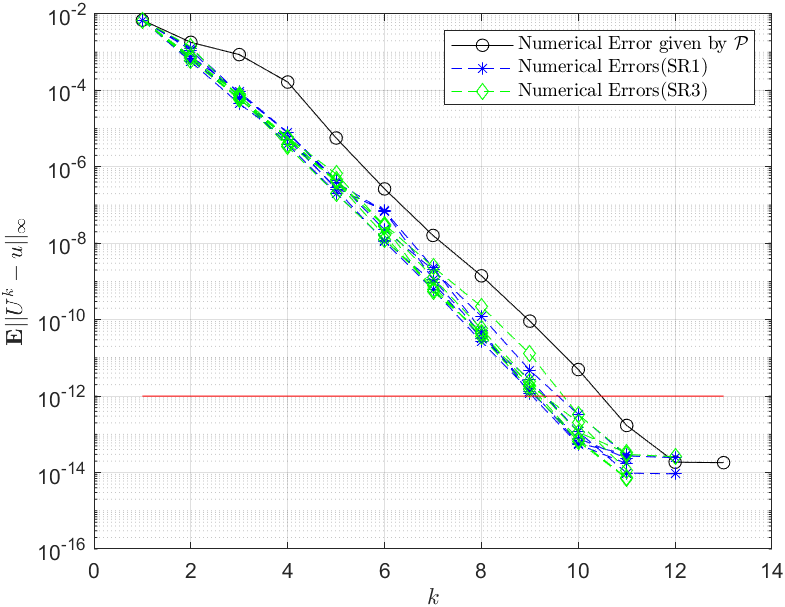}
        \label{Nolinear/Double-Well Potential2}
    }
         \subfigure[Sampling Rules 2 and 4]{
         \includegraphics[width=0.3\textwidth]{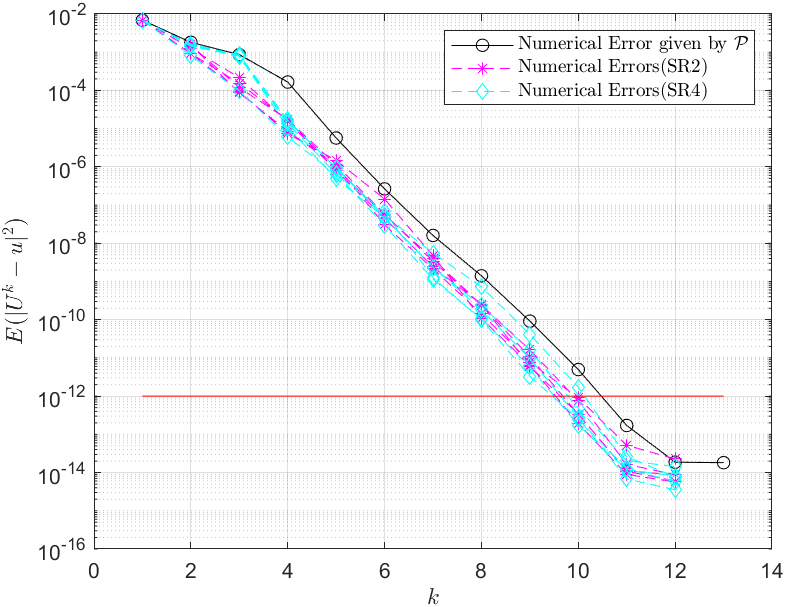}
        \label{Nolinear/Double-Well Potential3}
    }
    \captionsetup{font={footnotesize}}
    \caption{
   \subref{Nolinear/Double-Well Potential1} Numerical solutions obtained by $\mathcal{F}_{\Delta t}$ and $P_s$. \subref{Nolinear/Double-Well Potential2} Comparsion of $P$, $P_s$ with sampling rules 1 and 3. \subref{Nolinear/Double-Well Potential3} Comparsion of $P$,  
    $P_s$ with sampling rules 2 and 4. 
    }
    \label{Nolinear/Double-Well Potential}
    \end{figure}

When $\sigma = 20$, the impact of the stochastic term increases. We observe that both the $P$ algorithm and the $P_s$ algorithm based on the explicit Euler method  do not converge. We hypothesize that this is due to the explicit Euler method failing to converge, both in the strong and weak sense, when solving nonlinear stochastic differential equations (SDEs) \cite{Hutzenthaler2011EMdivergence}. Based on the methods proposed in \cite{Beyn2016projectedEM,Min2020projectedEMSDDEs,Min2021projectedEMjumps}, we introduce the following projected Euler method:

\begin{equation}
    \begin{aligned}
        \overline{U}_{n-1} &= \min\left(1, \delta t^{-\frac{1}{4}} \lvert U_{n-1} \rvert^{-1}\right) U_{n-1}, \\
        U_n &= \overline{U}_{n-1} + f(\overline{U}_{n-1}) \delta t + g(\overline{U}_{n-1}) \Delta W_{n},
    \end{aligned}
\end{equation}
where $W_n=W(t_{n})-W(t_{n-1})$ represents the increment of the Brownian motion.

We then investigate the convergence of the $P$ and $P_s$ algorithms based on the above projected Euler scheme. Figure \ref{Nolinear/Double-Well Potential PEM1} shows that the numerical solutions computed by $P_s$ and $\mathcal{F}_{\Delta t}$ are in close agreement. Figures \ref{Nolinear/Double-Well Potential PEM2} and \ref{Nolinear/Double-Well Potential PEM3} demonstrate that $P_s$ reaches the stopping criterion $\varepsilon$ faster than $P$. Specifically, for the same error tolerance, the number of iterations $k$ required by $P_s$ is smaller than that required by $P$.

     \begin{figure}[htbp]
    \subfigure[Numerical solutions]{
    \includegraphics[width=0.3\textwidth]{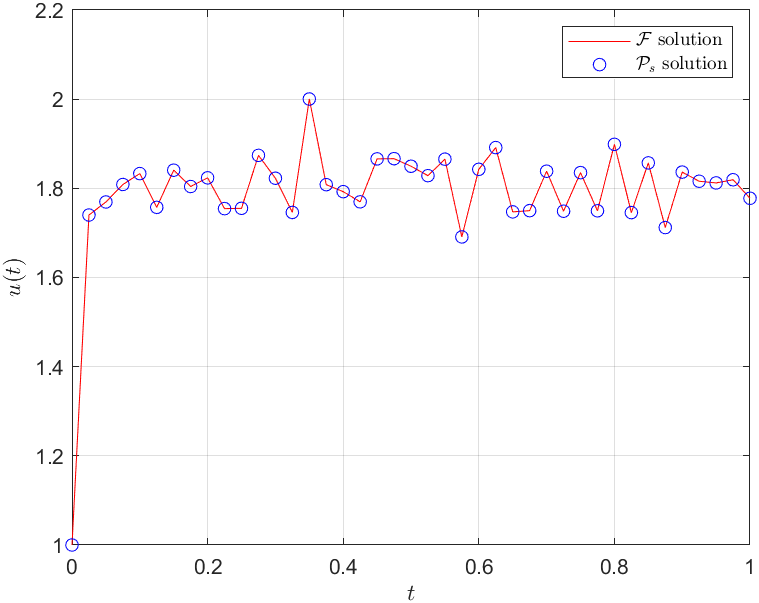}
     \label{Nolinear/Double-Well Potential PEM1}
    }
    \subfigure[Sampling Rules 1 and 3]{
    \includegraphics[width=0.3\textwidth]{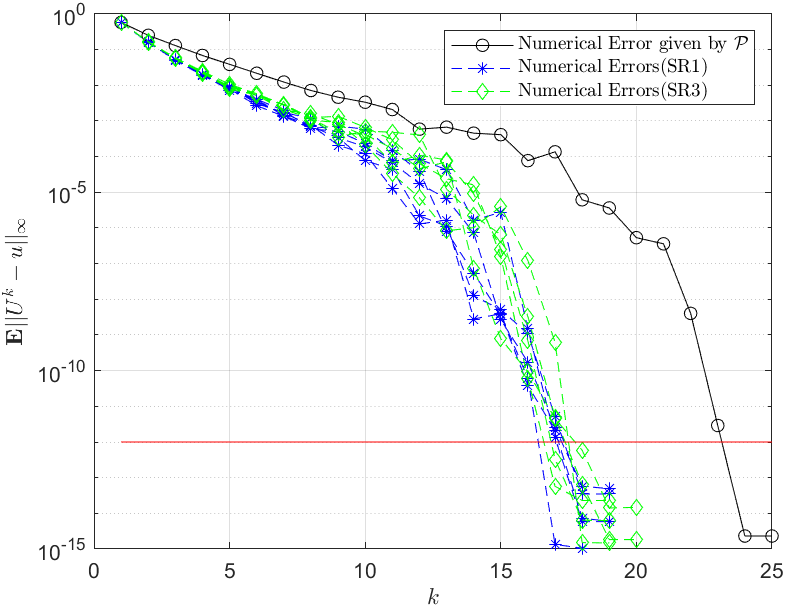}
     \label{Nolinear/Double-Well Potential PEM2}
    }
    \subfigure[Sampling Rules 2 and 4]{
    \includegraphics[width=0.3\textwidth]{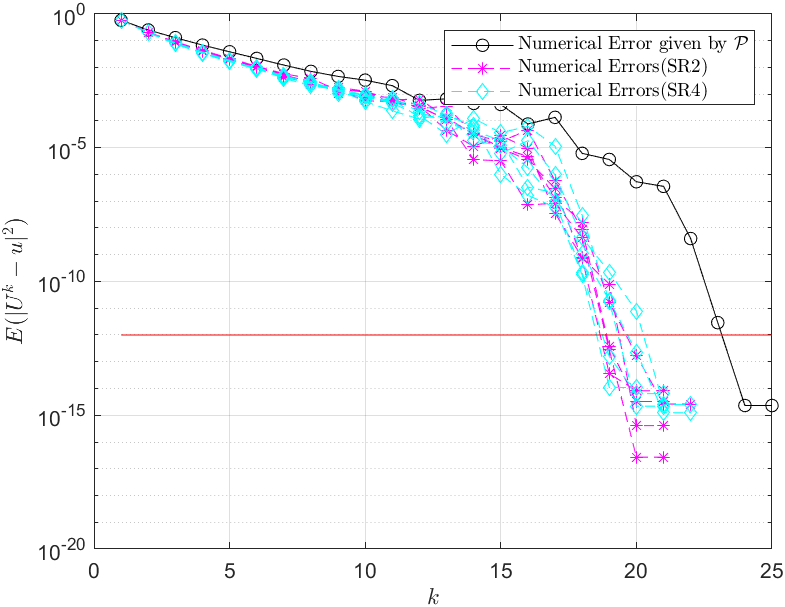}
    \label{Nolinear/Double-Well Potential PEM3}
    }
    \captionsetup{font={footnotesize}}
    \caption{
    \subref{Nolinear/Double-Well Potential PEM1} Numerical solutions obtained by $\mathcal{F}_{\Delta t}$ and $P_s$.
    \subref{Nolinear/Double-Well Potential PEM2} Comparsion of $P$, $P_s$ with sampling rules 1 and 3. 
    \subref{Nolinear/Double-Well Potential PEM3} Comparsion of $P$, $P_s$ with sampling rules 2 and 4. 
    }
    \label{Nolinear/Double-Well Potential}
    \end{figure}

\begin{example}
    \textbf{Population Dynamics Model} \cite{1999}\par
    The population dynamics in a natural environment can be modeled by the following stochastic differential equation:
    \begin{equation}
        dX(t) = r X(t) (K - X(t)) dt + \sigma X(t) dW(t),
    \end{equation}
    where $X(t)$ denotes the population density at time $t$. The constant $K > 0$ represents the environmental carrying capacity, while $r > 0$ is the intrinsic growth rate of the population. The exact solution of this system is given by:
    \begin{equation}\label{4.12}
        X(t) = \frac{X(0) e^{(rK - \frac{1}{2}\sigma^2)t + \sigma W(t)}}{1 + X(0)r \int_0^t e^{(rK - \frac{1}{2}\sigma^2)s + \sigma W(s)} ds}.
    \end{equation}
    From practical considerations, it is evident that population density cannot be negative. From the exact solution in \eqref{4.12}, it follows that if the initial condition satisfies $X(0) \geq 0$, then for any time $t > 0$, the solution also satisfies $X(t) \geq 0$.
\end{example}

In the numerical experiments, we set the initial condition $u(0) = 1$, with $N = 40$ and consider the time interval $T \in [0, 1]$. The coarse time step is chosen as $\Delta T = \frac{1}{40}$, and the fine time step as $\Delta t = \frac{1}{80}$. The remaining parameter values are set as $\theta_g = 0$, $\theta_f = 0$, $r = 0.5$, $K = 100$, and $\sigma = 0.05$.

Figure \ref{Nolinear/Population Dynamics Model1} demonstrates that the numerical solutions computed by $P_s$ and $\mathcal{F}_{\Delta t}$ are in close agreement. 
Furthermore, Figures \ref{Nolinear/Population Dynamics Model3} and \ref{Nolinear/Population Dynamics Model4} illustrate that $P_s$ reaches the stopping criterion $\varepsilon$ faster than $P$. Specifically, for the same error tolerance, the number of iterations $k$ required by $P_s$ is evident fewer than that required by $P$.

    \begin{figure}[htbp]
     \subfigure[Numerical solutions]{
        \includegraphics[width=0.3\textwidth]{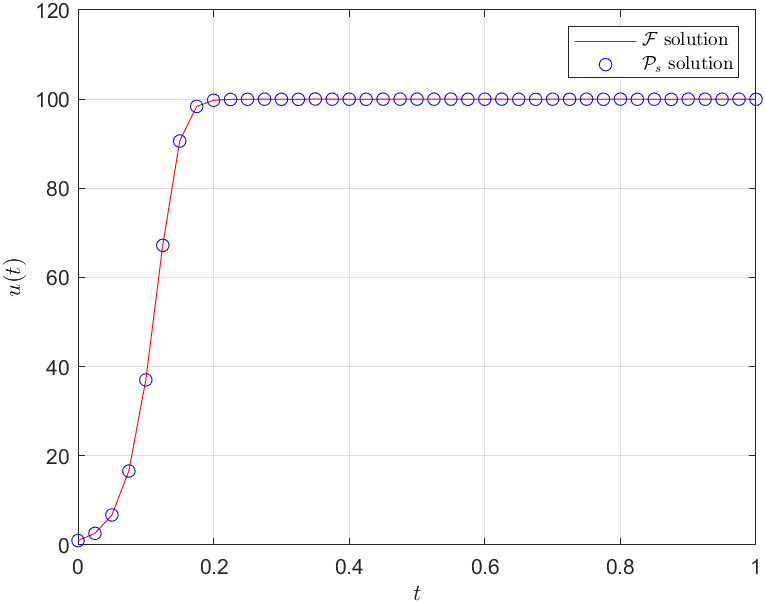}
        \label{Nolinear/Population Dynamics Model1}
    }
        \subfigure[Sampling Rules 1 and 3]{
        \includegraphics[width=0.3\textwidth]{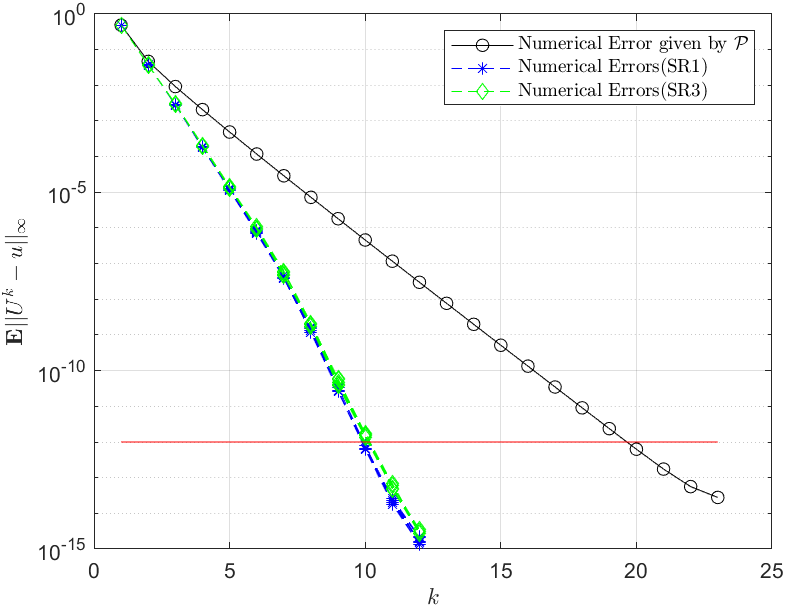}
        \label{Nolinear/Population Dynamics Model3}
    }
         \subfigure[Sampling Rules 2 and 4]{
         \includegraphics[width=0.3\textwidth]{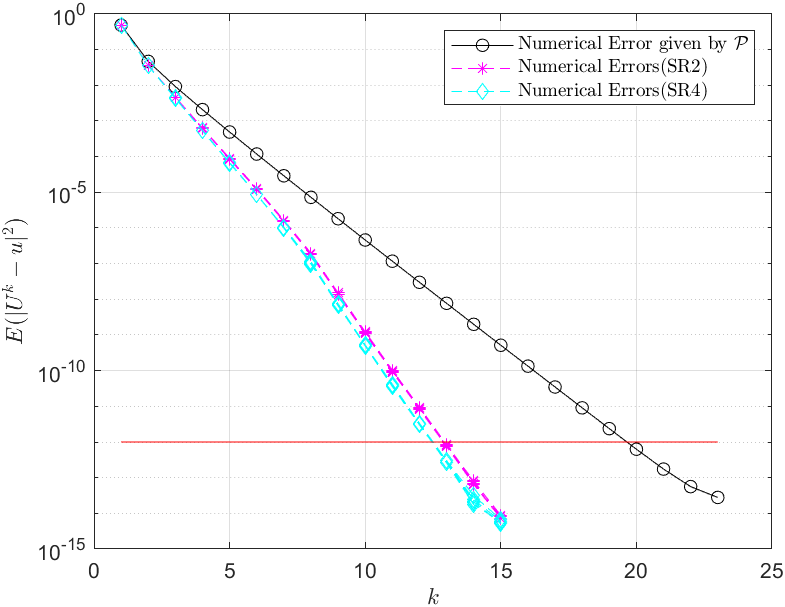}
       \label{Nolinear/Population Dynamics Model4}
    }
    \captionsetup{font={footnotesize}}
    \caption{
  \subref{Nolinear/Population Dynamics Model1} Numerical solutions obtained by $\mathcal{F}_{\Delta t}$ and $P_s$. \subref{Nolinear/Population Dynamics Model3} Comparsion of $P$, $P_s$ with sampling rules 1 and 3. 
  \subref{Nolinear/Population Dynamics Model4} Comparsion of $P$, $P_s$ with sampling rules 2 and 4. 
    \label{Nolinear/Population Dynamics Model}
    }
    \end{figure}

\section{Conclusions}

In this paper, we applied the SParareal algorithm to the Dahlquist test equation and analyzed its mean-square stability. Through theoretical analysis, we derived the conditions under which the algorithm maintains mean-square stability and established linear convergence bounds under different sampling rules. Similar to its performance for deterministic ordinary differential equations (ODEs), the SParareal algorithm demonstrates notable efficiency when applied to stochastic differential equations (SDEs). Compared to the classical Parareal algorithm, the SParareal algorithm exhibits faster convergence; specifically, for the same error tolerance, the number of iterations $k$ required by $P_s$ is evident fewer than that required by $P$. Numerical experiments further confirm the effectiveness of the SParareal algorithm for both linear and nonlinear SDEs.

Theoretically, the number of iterations $k$ required by the SParareal algorithm decreases as the number of samples $m$ increases. However, numerical results reveal that beyond a certain threshold of $m$, the reduction in the number of iterations $k$ becomes less apparent. Although the SParareal algorithm achieves improved efficiency with increasing $m$, it also demands greater computational resources. Therefore, as highlighted in \cite{2021Traffic}, balancing efficiency and computational cost by identifying an optimal value of $m$ remains a critical area for future research.   

Furthermore, based on the numerical observations in Example 3, the study of the convergence of the $P$ and $P_s$ algorithms for nonlinear stochastic differential equations, utilizing modified versions of the Euler method (such as the projected Euler method \cite{Beyn2016projectedEM}, the taming Euler method \cite{Hutzenthaler2012tamedEM,Gao2024tamedEM}, and the truncated Euler method \cite{Mao2015truncatedEM}), represents an important direction for future research.


\section*{Funding}
This work was supported by  National Natural Science Foundation of China (Nos. 12201586)  and China 
	Postdoctoral Science Foundation (Grant No. 2021M703008),  the Guangdong Provincial Natural Science Foundation of China (Grant No. 2023A1515010809), the Fundamental Research Funds for the Central Universities, China University of Geosciences (Wuhan, Grant numbers: CUG2106127 and CUGST2).

\section*{Data availability}
Data sharing not applicable to this article as no datasets were generated or analyzed during the current study.

\section*{Declaration}
\textbf{Ethical Approval} Not Applicable. \\
\textbf{Conflict of interest} The authors declare no competing interests.

\bibliography{sn-Ref}


\end{document}